\documentclass[11pt]{amsart}
\usepackage{amssymb,latexsym,amsmath,amsthm,enumitem,mathrsfs,geometry}
\usepackage{fullpage}
\usepackage{fancyhdr}
\usepackage{color}
\usepackage{stmaryrd}
\usepackage{mathrsfs}
\usepackage{hyperref}
\usepackage{lineno}
\usepackage{diagbox}
\usepackage{array}
\usepackage{tikz}
\usetikzlibrary{arrows}
\usetikzlibrary{positioning}
\usepackage{float}
\usepackage{bbm}
\usepackage{cleveref}
\usepackage{dsfont}
\usepackage{graphicx}
\usepackage{comment}

\usepackage{url}
\usepackage{caption}

\newcommand{\1}{\mathbbm{1}}

\newtheorem{theorem}{Theorem}
\newtheorem{lem}{Lemma}
\newtheorem{corollary}{Corollary}
\newtheorem*{theorem*}{theorem}
\newtheorem*{AH}{The Alternative Hypothesis (AH)}
\newtheorem*{MT}{Montgomery Theorem (MT)}
\newtheorem*{MTSP}{Montgomery Theorem for Sums over Pairs (MT-Pairs)}
\newtheorem*{AH2}{The Alternative Hypothesis for Differences of Zeros (AH-Pairs)}
\newtheorem*{AH2Strong}{The Strong Alternative Hypothesis for Differences of Zeros (Strong AH-Pairs)}
\newtheorem*{ESH}{The Essential Simplicity Hypothesis (ESH)}
\newtheorem*{AH3}{The Alternative Hypothesis for Density of Zeros Spacings (AH-Density)}
\theoremstyle{remark}
\newtheorem{remark}{Remark}

\definecolor{pink}{rgb}{1,.2,.6}
\definecolor{orange}{rgb}{0.7,0.3,0}
\definecolor{blue}{rgb}{.2,.6,.75}
\definecolor{green}{rgb}{.4,.7,.4}
\definecolor{purple}{RGB}{127,0,255}

\numberwithin{equation}{section}

\allowdisplaybreaks

\begin{document}

\title[Alternative Hypothesis]{The Alternative Hypothesis for Zeros of the Riemann Zeta-Function}

\author[Baluyot]{Siegfred Alan C. Baluyot}
\address{Department of Mathematics, Eastern Carolina University}
\email{baluyots24@ecu.edu}

\author[Goldston]{Daniel Alan Goldston}
\address{Department of Mathematics and Statistics, San Jose State University}
\email{daniel.goldston@sjsu.edu}

\author[Suriajaya]{Ade Irma Suriajaya}
\address{Faculty of Mathematics, Kyushu University}
\email{adeirmasuriajaya@math.kyushu-u.ac.jp}

\author[Turnage-Butterbaugh]{Caroline L. Turnage-Butterbaugh}
\address{Department of Mathematics and Statistics, Carleton College}
\email{cturnageb@carleton.edu}


\date{\today}

\begin{abstract}
In 2016, the first-named author introduced a formulation of the Alternative Hypothesis that assumes that consecutive zeros of the Riemann zeta-function are spaced at multiples of half of the average spacing, but does not assume that the zeros are simple. In this paper, we assume the Riemann Hypothesis and a similar formulation of the Alternative Hypothesis, and for each integer $k$ we obtain constraints on the density of pairs of zeros whose normalized differences are at $k/2$ times the average spacing. These constraints, in turn, restrict the density of (possible) multiple zeros. We also formulate a stronger version of the Alternative Hypothesis and show that it implies the Essential Simplicity Hypothesis.
\end{abstract}

\maketitle


\section{Introduction and Main Results}
\label{sec:1}

The Alternative Hypothesis (AH) is an alternative to the pair correlation conjecture for zeros of the Riemann zeta-function  and arose as a consequence of the (possible) existence of Landau-Siegel zeros. In a 1996 lecture, Heath-Brown \cite{Heath-Brown96} discussed that if there is a sequence of real Dirichlet $L$-functions with Landau-Siegel zeros, then (i) there exist infinitely many extremely long ranges in the critical strip where all the zeros of the Riemann zeta-function are on the critical line, i.e., the Riemann Hypothesis is true in this range, (ii) the zeros are all simple, and (iii) the zeros all occur at spacings that are nearly exactly integer multiples of half the average spacing between zeros. These hypothetical properties of the zeros of the Riemann zeta-function, however, need not depend on the existence of Landau-Siegel zeros, and the spacing hypothesis was later referred to as the \textit{Alternative Hypothesis}.
The value of AH is that it provides a simple distribution of zeros that is completely counter to the experimental evidence supporting the pair correlation conjecture, and yet cannot be disproved at present. 

In this introduction we discuss our main results and their associated formulations of the Alternative Hypothesis. In Section \ref{Sec2} we present a more detailed description of the main results. To begin, let $\rho = \beta +i\gamma, \,\beta, \gamma \in \mathbb{R}$ denote a non-trivial zero of the Riemann zeta-function. For $T\ge 3$, we have
\begin{equation} \label{N(T)} N(T) :=\sum_{0<\gamma \le T}1 = \frac{T}{2\pi}\log \frac{T}{2\pi} - \frac{T}{2\pi} + O(\log T), \end{equation}
where the zeros are counted with multiplicity. We almost always assume the Riemann Hypothesis (RH) in this paper, i.e. we often assume that $\rho = 1/2+i\gamma$. Consider the sequence of ordinates
\[
0 < \gamma_1 < \gamma_2 < \cdots < \gamma_n < \cdots,
\]
where we only consider zeros in the upper half-plane since the zeros are symmetric about the real line. By \eqref{N(T)}, the average spacing between the zero $1/2+i\gamma$ and the next higher zero is $2\pi/\log \gamma$. Therefore, we may denote the normalized ordinate of a zero by
\[ 
\widetilde{\gamma} := \frac{\gamma}{2\pi} \log \gamma, 
\]
and note that the average spacing between consecutive $\widetilde{\gamma}$'s is asymptotic to $1$. Within this framework, the first-named author gave the following formulation of the Alternative Hypothesis in \cite{Bal16}.
\begin{AH} \label{AH}
For each integer $n\ge 1$ there exists an integer $k_n\ge 0$ with 
\begin{equation*} 
\widetilde{\gamma}_{n+1} - \widetilde{\gamma}_{n} = \frac12 k_n + O\big(\,|\gamma_{n+1} - \gamma_n|\cdot\psi(\gamma_n)\,\big),
\end{equation*}
where $\psi(\gamma)$ is a positive function such that $\psi(\gamma) \to \infty$ and $\psi(\gamma) = o(\log \gamma)$ as $\gamma \to \infty$.
\end{AH}

One goal of \cite{Bal16} was to determine consequences of this hypothesis on the pair correlation of zeros. For that purpose, \hyperref[AH]{\textbf{AH}} was reformulated in \cite[Lemma 5.1]{Bal16} for pairs of zeros that are not necessarily consecutive. In this paper, we use the following modified version of the Alternative Hypothesis for differences of zeros. This version of the Alternative Hypothesis is implied by the statement of \hyperref[AH]{\textbf{AH}} given above and is similar to \cite[Lemma 5.1]{Bal16}, which we state as Lemma \ref{SiegLem} in the next section. 
 
\begin{AH2}\label{AH2}
Suppose $M$ is a positive real number that we can take as large as we wish, and let $R(T)$ be a positive decreasing function such that $R(T) \to 0$ as $T\to \infty$. Define
\begin{equation} \label{P(T,M)} \mathcal{P}(T,M) := \left\{ (\gamma,\gamma'): \frac{T}{\log^2T} < \gamma , \gamma'\le T , \quad \left|\frac{\gamma-\gamma'}{2\pi}\log T\right| \le M\right\}. \end{equation}
Then for every $(\gamma,\gamma') \in \mathcal{P}(T,M)$ there is an integer $k$ such that
\begin{equation}\label{AH2k} (\gamma-\gamma')\frac{\log T}{2\pi} =\frac{ k}{2} + O((|k|+1) R(T)). \end{equation}
\end{AH2} 
\noindent We see that $k\ll M$ since \eqref{P(T,M)} and \eqref{AH2k} imply $M \gg k + o(k) \gg k$.

\begin{remark} \label{comment2} The case $k=0$ is of particular interest independent of \hyperref[AH]{\textbf{AH}}. This can occur in two ways. First, it occurs for pairs $(\gamma, \gamma')$ when $\gamma = \gamma'$, which for a simple zero occurs once, and in general $m_\gamma^2$ times for a zero of multiplicity $m_\gamma$. Second, $k=0$ can occur for different distinct zeros which are very close together. 
\end{remark}

\begin{remark} Our results are asymptotic, and therefore would still hold if the statement of \hyperref[AH]{\textbf{AH}} was modified to allow for an exceptional set of zeros that do not satisfy \hyperref[AH]{\textbf{AH}}.
\end{remark}

\begin{remark}
Notice that \hyperref[AH]{\textbf{AH}} and \hyperref[AH2]{\textbf{AH-Pairs}} do not directly depend on RH since they are statements on the imaginary part of the zeros alone.
Without RH, in \Cref{comment2}, $\gamma =\gamma'$ can also occur for zeros on the same horizontal line, and further, for $B_{k/2}$ defined below in \eqref{Bk/2}, $(\gamma,\gamma')\in B_{k/2}$ refers to pairs of zeros in a horizontal strip.
One needs to interpret such unconditional results carefully since $\gamma =\gamma'$ for different zeros does not require the existence of multiple zeros. In stating our results concerning \hyperref[AH2]{\textbf{AH-Pairs}} we will be careful to only assume RH in places where it is needed. 
\end{remark}

We aim to count the number of pairs of zeros in $\mathcal{P}(T,M)$ with normalized difference closest to the half-integer $k/2$. Following \cite{Bal16}, we define $B_{k/2}$ for any fixed number $0<\delta\le1/2$ by \begin{equation} \label{Bk/2}
B_{k/2} = B_{k/2}(T,M,\delta) := \left\{(\gamma,\gamma') \in \mathcal{P}(T,M): \frac{k}2 - \frac{\delta}2 < \frac{\gamma-\gamma'}{2\pi}\log T \le \frac{k}2 + \frac{\delta}2\right\}, \end{equation}
which by \eqref{AH2k} for large enough $T$ is well-defined and partitions $\mathcal{P}(T,M)$ over the half-integers $|k/2|\le M+\delta/2$. We define the \lq\lq density" of pairs closest to $k/2$ by
\begin{equation*} P_{k/2} = P_{k/2}(T) := \left(\frac{T}{2\pi}\log T\right)^{-1}|B_{k/2}|. \end{equation*}
Note that for all $k\in\mathbb{Z}$ we have
\begin{equation*} 
P_{k/2} = P_{-k/2}.
\end{equation*}
Our first result is as follows.

\begin{theorem} \label{thm1}
Assume the Riemann Hypothesis and \hyperref[AH2]{AH-Pairs}. As $T\to\infty$, we have 
\begin{equation} \label{thm1a} 1+o(1) \le P_0 \le \frac32 - \frac{2}{\pi^2} + o(1), \end{equation}
and for $k\in \mathbb{Z}$ and $k\neq 0$ we have 
\begin{equation} \label{thm1b} P_{k/2} \sim \begin{cases}
P_0-\frac12, & \text{if $k\neq 0$ is even,} \\
\frac32-\frac{2}{\pi^2k^2}-P_0, & \text{if $k$ is odd.}
\end{cases}
\end{equation}
\end{theorem}

\begin{remark} We see on RH and \hyperref[AH2]{\textbf{AH-Pairs}} that $\limsup_{T\to \infty} P_0 \le 3/2 - 2/\pi^2 = 1.29735\ldots$ while on RH alone the best result known \cite{CGL2020} is that $\limsup_{T\to \infty}P_0\le 1.3208$.
\end{remark}

\Cref{thm1} implies that if one of the ``limiting densities" 
\begin{equation} \label{pk/2}
p_{k/2} := \lim_{T\to\infty} P_{k/2}(T)
\end{equation}
exists, then all of them exist. In other words, one limiting density $p_{k_0/2}$ for a specific $k_0$ determines the values of $p_{k/2}$ for all $k\in\mathbb{Z}$. We state this precisely as Corollary \ref{cor1} in the next section.

We expect that $p_0=1$, which under RH, is a reformulation of the following conjecture, see \cite{Mue83}.
\begin{ESH}\label{ESH}
Almost all the zeros of the Riemann zeta-function are simple. Moreover, almost all of the distinct zeros are not spaced arbitrarily closer together than the average spacing. 
\end{ESH}
Without RH, the term {\it spacing} in the statement of \hyperref[ESH]{\textbf{ESH}} refers to the {\it vertical difference} between two zeros, which may be far apart from one another horizontally. Therefore, without RH, the statement of \hyperref[ESH]{\textbf{ESH}}  immediately implies $p_0=1$, however the converse is not necessarily true. Our second main result addresses this and shows that RH and a stronger form of \hyperref[AH2]{\textbf{AH-Pairs}} imply that $p_0=1$, which in turn implies \hyperref[ESH]{\textbf{ESH}} . The required stronger form of \hyperref[AH2]{\textbf{AH-Pairs}} is as follows.
\begin{AH2Strong} \label{AH2Strong}
Suppose $\mathcal{M}$ is a positive real number which we can take as large as we wish, and let $R(T)$ be a positive decreasing function such that $ R(T)\log T \to 0$ as $T\to \infty$. Define
\begin{equation*} 
\mathcal{Q}(T,\mathcal{M}) := \left\{ (\gamma,\gamma'): \frac{T}{\log^2T} < \gamma , \gamma'\le T , \quad \left|\gamma-\gamma'\right| \le\mathcal{M}\right\}. \end{equation*}
Then for every $(\gamma,\gamma') \in \mathcal{Q}(T,\mathcal{M})$ there is an integer $k\ll \mathcal{M}\log T$ such that
\begin{equation*} 
(\gamma-\gamma')\frac{\log T}{2\pi} =\frac{ k}{2} + O((|k|+1) R(T)). \end{equation*}
\end{AH2Strong} 
We extend the definition of $B_{k/2}$ in \eqref{Bk/2} to pairs $(\gamma,\gamma') \in \mathcal{Q}(T,\mathcal{M})$ with the same conditions, and partition $\mathcal{Q}(T,\mathcal{M})$ over half-integers $|k/2|\ll \mathcal{M}\log T$.
\begin{theorem} \label{thm3-simple_zeros}
Assume that the Riemann Hypothesis and the \hyperref[AH2Strong]{Strong AH-Pairs} hold. Then  
\begin{equation*} 
p_0 = \lim_{T\to \infty} P_0 = 1, \end{equation*}
and thus the Essential Simplicity Hypothesis is true.
\end{theorem}

Our results are based on Montgomery's work on pair correlation of zeros of the Riemann zeta-function. Montgomery \cite{Montgomery73} assumed RH, and defined, for real $\alpha$,
\begin{equation} \label{Falpha}
F(\alpha) := \left(\frac{T}{2\pi}\log T\right)^{-1} \sum_{0<\gamma,\gamma'\le T} T^{i\alpha(\gamma-\gamma')} w(\gamma-\gamma'), \quad\text{ where }~ w(u) = \frac{4}{4+u^2}.
\end{equation}
Here the sum is over the imaginary parts of pairs of zeros $\rho = 1/2+i\gamma$ and $\rho'=1/2 +i\gamma'$ in the critical strip.
Montgomery  evaluated $F(\alpha)$ for $|\alpha|\le 1$ and conjectured its asymptotic behavior for $|\alpha|>1$. We now state Montgomery's theorem which incorporates some improvements from \cite{GM87}, see also \cite[Montgomery Theorem (MT)]{BGST-CL}.

\begin{MT} \label{MonThm} Assume the Riemann Hypothesis. The function $F(\alpha)$ is real, even, and nonnegative. Moreover, as
$T \to \infty$, we have
\begin{equation}\label{MTeq} F(\alpha) =  T^{-2\alpha}\log T \left(1+ O\left(\frac{1}{\sqrt{\log T}}\right) \right) +
\alpha+ O\left(\frac{1}{\sqrt{\log T}}\right)
\end{equation}
uniformly for $0\le \alpha \le 1$.
\end{MT}
\begin{remark}
We have removed an extraneous factor of $\log\log T$ from \cite{GM87} which can be avoided using Lemma 6 there in place of Lemma 7. The statement there has also been corrected slightly. See also \cite{Gold81} and \cite{LPZ17}. Montgomery and Vaughan in the forthcoming book Multiplicative Number Theory III have obtained a significantly refined version of \hyperref[MonThm]{\textbf{MT}}. Of particular interest, when $\alpha=1$, they prove that the error term $O(1/\sqrt{\log T})$ above can be replaced with $O(\log\log T /\log T)$.
\end{remark}
\begin{remark}
In our previous paper \cite{BGST-PC}, we proved that \hyperref[MonThm]{\textbf{MT}} holds unconditionally without assuming RH for a generalization of $F(\alpha)$ defined by 
\begin{equation*} 
F(\alpha) := \left(\frac{T}{2\pi}\log T\right)^{-1} \sum_{\substack{\rho, \rho' \\ 0<\gamma,\gamma' \le T}} T^{\alpha(\rho-\rho')}W(\rho-\rho'), \qquad \text{where} \qquad W(u) := \frac{4}{4 - u^2}.
\end{equation*}
The result remains the same even if we modify the summation to be over $T<\gamma,\gamma' \le 2T$, see \cite{BGST-CL}.
\end{remark}We note that \hyperref[AH2]{\textbf{AH-Pairs}} is consistent with \hyperref[MonThm]{\textbf{MT}}, however it determines a completely different behavior of $F(\alpha)$ when $|\alpha|>1$, as shown below in Figure \ref{fig-AH-triangles}.

\begin{figure}[H]
\centering
\includegraphics[width=0.65\textwidth]{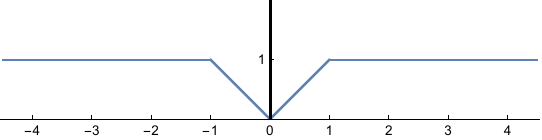}
\vskip0.5in
\includegraphics[width=0.65\textwidth]{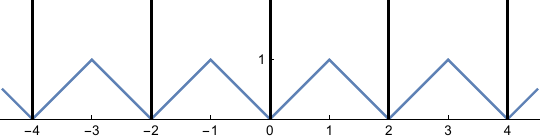}
\vskip0.5in
\includegraphics[width=0.65\textwidth]{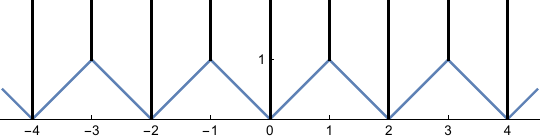}
\vskip0.25in
\caption{Three plots of $F(\alpha)$ under different assumptions. The top plot assumes the GUE model, the middle plot assumes AH with $p_0=1$, and the bottom plot assumes AH with $P_0>1$ and that if $p_0$ exists then $p_0 >1$.}
\label{fig-AH-triangles}
\end{figure}

In recent work, Lagarias and Rodgers \cite{LR20} proved that \hyperref[AH]{\textbf{AH}} is compatible not only with Montgomery's unconditional results on $F(\alpha)$ for $|\alpha|\le 1$, but for all band-limited higher correlations. Their work uses the theory of point processes, and they work under the assumption that almost all of the points of the point process are distinct. This assumption corresponds to assuming $p_0 =1$ and the middle plot of $F(\alpha)$ given in Figure \ref{fig-AH-triangles}.

If $p_0>1$ then $F(\alpha)$ behaves as depicted in the bottom plot in Figure \ref{fig-AH-triangles}, which is obtained by modifying the middle plot to include delta functions with mass $2(P_0 -1)$ at the odd integers. The precise behavior of $F(\alpha)$ in this setting is stated in the next section as Theorem \ref{thm4-r(k)p_k}.

We now formulate the following conjecture which, upon assuming the limiting densities $p_{k/2}$ in \eqref{pk/2} exist, is a consequence of RH and \hyperref[AH2]{\textbf{AH-Pairs}}.

\begin{AH3} \label{AH-density}
The limiting densities $p_{k/2}$ exist and satisfy
\begin{equation*} 
1 \le p_0 \le \frac32 - \frac{2}{\pi^2}, \end{equation*}
and for $k\in \mathbb{Z}$ 
\begin{equation*}
p_{k/2} = \begin{cases}
p_0-\frac12, & \text{if $k\neq 0$ is even}, \\
\frac32-\frac{2}{\pi^2k^2}-p_0, & \text{if $k$ is odd}.
\end{cases} \end{equation*}
\end{AH3}

As we see in Corollary \ref{cor1} below, if  $p_0$ exists then all of the densities $p_{k/2}$ exist. \hyperref[AH-density]{\textbf{AH-Density}} thus contains nearly the same information provided by \hyperref[MonThm]{\textbf{MT}} and \hyperref[AH2]{\textbf{AH-Pairs}}. As an example of this, in \Cref{thm2-generalizedMTwith_p_k} we prove a version of \hyperref[MonThm]{\textbf{MT}}  from \hyperref[AH-density]{\textbf{AH-Density}}.

\section{Detailed Description of the Main Results}\label{Sec2}

We begin with the following lemma from \cite[Lemma 5.1]{Bal16} which is a reformulation of \hyperref[AH]{\textbf{AH}} for pairs of zeros which may not be consecutive.
\begin{lem}\label{SiegLem}
Suppose $M$ is a positive real number which we can take as large as we wish, and let $\mathcal{P}(T,M)$ be as defined in \eqref{P(T,M)}.
Suppose the Alternative Hypothesis is true. Let $\Psi(T)$ be a function with $\Psi(T)\to \infty$ as $T\to \infty$ and $\Psi(T)= o(\log T)$. Then for every $(\gamma,\gamma') \in \mathcal{P}(T,M)$ there is an integer $k$ such that
\[ (\gamma-\gamma')\frac{\log T}{2\pi} = \frac{k}{2} + O\left(|\gamma-\gamma'|\left(\Psi(T) + \frac{M\log T}{T} +\log \log T\right)\right).\]
Here we may choose the integer $k$ so that it satisfies
\begin{equation} \label{k<<M}
k\ll M.
\end{equation}
\end{lem}

Except for \eqref{k<<M}, this is \cite[Lemma 5.1]{Bal16}. 
To see that \eqref{k<<M} holds, observe that for $\gamma,\gamma'\in\mathcal{P}(T,M)$,
if
$$
\left| \frac{\gamma-\gamma'}{2\pi} \log T\right| \le T,
$$
then since $\Psi(T)=o(\log T)$ and $|\gamma-\gamma'|\log T\ll M$ we have
\begin{align*}
|\gamma-\gamma'|\left( \Psi(T) +\frac{M\log T}{T}+\log\log T\right) 
&\ll |\gamma-\gamma'| \left( \log T +\frac{M\log T}{T}\right) \\
&= |\gamma-\gamma'|\log T + ( |\gamma-\gamma'|\log T ) \frac{M}{T} \\
&\ll M + T\cdot\frac{M}{T} \ll M.
\end{align*}
Thus any $k$ satisfying
$$
(\gamma-\gamma')\log T = \pi k + O\left( |\gamma-\gamma'|\left( \Psi(T)+\frac{M\log T}{T}+\log\log T\right) \right)
$$
must also satisfy $k\ll M$. On the other hand, if
$$
\left| \frac{\gamma-\gamma'}{2\pi} \log T\right| \geq T,
$$
then \textit{any} choice of an integer $k\ll M$ must also satisfy
$$
(\gamma-\gamma')\log T =\pi k + O\left( |\gamma-\gamma'|\left( \Psi(T) +\frac{M\log T}{T}+\log\log T\right) \right)
$$
because in this case
\[
|\gamma-\gamma'|\left( \Psi(T) +\frac{M\log T}{T}+\log\log T\right) 
\geq |\gamma-\gamma'| \frac{M\log T}{T}
\geq (2\pi T) \frac{M}{T} \gg M.
\]

Lemma \ref{SiegLem} is the basis of our formulation of \hyperref[AH2]{\textbf{AH-Pairs}}.
Note that by Lemma \ref{SiegLem}, \hyperref[AH]{\textbf{AH}} implies 
\[ R(T) = \frac{\Psi(T)}{\log T} + \frac{M}{T} +\frac{\log\log T}{\log T} . \] 
Since $\Psi(T)=o(\log T)$, without further information on $\Psi(T)$ this matches the requirement in \hyperref[AH2]{\textbf{AH-Pairs}} that $R(T)\to 0$ as $T\to \infty$.
In addition, unlike Lemma \ref{SiegLem}, \hyperref[AH2]{\textbf{AH-Pairs}} treats the case when $k=0$ the same as the cases when $k\neq 0$.

In proving our results we do not directly use \hyperref[MonThm]{\textbf{MT}}, but rather the following nearly immediate consequence of \hyperref[MonThm]{\textbf{MT}}. Recall that a function $f(x)$ is Lipschitz continuous at a point $x=a$ if there are constants $C>0$ and $\delta >0$ such that $|f(x)-f(a)| \le C |x-a|$ for all $x$ in a neighborhood $|x-a|< \delta$ of $a$. We can also define left or right Lipschitz continuous at $x=a$ in neighborhoods $a-\delta < x \le a$ or $a\le x < a+\delta$ respectively.
\begin{MTSP}
\label{lemsumpair} Let $g(\alpha)\in L^{\!1}(\mathbb R)$, and define the Fourier transform $\widehat{g}(t)$ of $g(\alpha)$ by
\begin{equation} \label{hat-g}
\widehat{g}(t) = \int_{-\infty}^\infty g(\alpha)e(-t\alpha)\,d\alpha, \qquad \text{where} \qquad e(\theta) = e^{2\pi i \theta}.
\end{equation}
Assuming the Riemann Hypothesis, we have
\begin{equation} \label{Sumpair}
\sum_{0<\gamma,\gamma'\le T }\widehat{g}\Big(\frac{\gamma-\gamma'}{2\pi}\log T\Big)w(\gamma-\gamma')
= \left(\frac{T}{2\pi}\log T \right) \int_{-\infty}^\infty F(\alpha)g(\alpha)\,d\alpha .
\end{equation}
If we assume in addition that $g(\alpha)$ is even with support in $|\alpha|\le 1$, and $g(\alpha)$ is Lipschitz continuous at $\alpha =0$, then 
\begin{equation} \label{g-transform} \widehat{g}(t) = 2\operatorname{Re} \int_0^1 g(\alpha) e(t\alpha)\, d\alpha , \end{equation}
and
\begin{equation} \label{pairsum}
\sum_{0<\gamma,\gamma'\le T }\widehat{g}\Big(\frac{\gamma-\gamma'}{2\pi}\log T\Big)w(\gamma-\gamma')
= \frac{T}{2\pi}\log T \left( 
g(0) + 2 \int_0^1\alpha g(\alpha)\, d\alpha +O\left(\frac{1}{\sqrt{\log T}}\right) \right).
\end{equation}
\end{MTSP}

\begin{remark}
While \eqref{pairsum} is only true if we assume RH, we can evaluate the sum over pairs of zeros on the left-hand side of \eqref{pairsum} by using \hyperref[AH2]{\textbf{AH-Pairs}} without assuming RH, but then we cannot conclude the result agrees with the right-hand side of \eqref{pairsum} without assuming RH.
\end{remark}

Turning now to the evaluation of the left-hand side of \eqref{pairsum} using \hyperref[AH2]{\textbf{AH-Pairs}}, by \eqref{N(T)} the average spacing between consecutive zeros of height $\le T$ is $2\pi/\log T$, and thus the \hyperref[AH2]{\textbf{AH-Pairs}} conjecture implies all the pairs of zeros in $\mathcal{P}(T,M)$ are extremely close to multiples of half of the average spacing.
\Cref{thm1} is obtained by making use of \hyperref[MonThm]{\textbf{MT}} in the form of \hyperref[lemsumpair]{\textbf{MT-Pairs}}.
As a corollary to \Cref{thm1}, \hyperref[AH2]{\textbf{AH-Pairs}} allows us to determine the limiting behavior of the density functions.
Recall the ``limiting densities"
\begin{equation*}
p_{k/2} = \lim_{T\to\infty} P_{k/2}(T)
\end{equation*}
as defined in 
\eqref{pk/2}.

\begin{corollary} \label{cor1}
Assume the Riemann Hypothesis and \hyperref[AH2]{AH-Pairs}. In addition, assume that the limiting density $p_0$ exists. Then
\begin{equation*} 
1 \le p_0 \le \frac32 - \frac{2}{\pi^2}, \end{equation*}
and for $k\in \mathbb{Z}$, the limiting densities $p_{k/2}$ all exist, and
\begin{equation*}
p_{k/2} = \begin{cases}
p_0-\frac12, & \text{if $k\neq 0$ is even}, \\
\frac32-\frac{2}{\pi^2k^2}-p_0, & \text{if $k$ is odd}.
\end{cases} \end{equation*}
The assumption that $p_0$ exists may be replaced with the assumption that $p_{k/2}$ exists for any one value of $k$ and that will then imply all the densities exist for all $k$.
\end{corollary}

This corollary follows immediately from \Cref{thm1} since the bounds on $p_0$ follow from \eqref{thm1a}, and by taking limits in \eqref{thm1b} the existence of $p_0$ implies the existence of each $p_{k/2}$, while the existence on one $p_{k/2}$ implies the existence of $p_0$ and thus the remaining $p_{k/2}$.
We also give two further corollaries to \Cref{thm1} as the two extreme cases of \hyperref[AH2]{\textbf{AH-Pairs}}.

\begin{corollary} \label{cor2} Assuming the Riemann Hypothesis, \hyperref[AH2]{AH-Pairs}, and $p_0 =1$, we have 
\begin{equation*} p_{k/2} = \begin{cases}
1, & \text{if $k=0$}, \\
\frac12, & \text{if $k\neq 0$ is even}, \\
\frac12-\frac{2}{\pi^2k^2}, & \text{if $k$ is odd}.
\end{cases} \end{equation*}
\end{corollary}

\begin{corollary} \label{cor3} Assume the Riemann Hypothesis, \hyperref[AH2]{AH-Pairs}, and that $p_0=\frac32-\frac{2}{\pi^2}$. Then we have 
\begin{equation*} p_{k/2} =
\begin{cases}
\frac32-\frac2{\pi^2}, &\text{if $k=0$}, \\
1- \frac{2}{\pi^2}, &\text{if $k\neq 0$ is even}, \\
\frac{2}{\pi^2}\left(1-\frac{1}{k^2}\right), &\text{if $k$ is odd}.
\end{cases}
\end{equation*}
Note here that $p_{1/2}=0$.
\end{corollary}

We now outline the proof of \Cref{thm1}. 
We need to evaluate the sum over $\gamma-\gamma'$ in \hyperref[lemsumpair]{\textbf{MT-Pairs}} using \hyperref[AH2]{\textbf{AH-Pairs}}. Since \hyperref[AH2]{\textbf{AH-Pairs}} applies to terms $(\gamma,\gamma') \in \mathcal{P}(T,M)$, we first need to estimate the error in removing the terms where $(\gamma,\gamma')\not \in \mathcal{P}(T,M)$. 

\begin{lem} \label{sumerror} Let $b> 1$ be a constant, and suppose $r(t)\in L^1(\mathbb{R})$, $|r(t)| \ll 1$ for all $t$, and that there is a constant $C(r)$ depending on $r$ for which
\begin{equation*} 
|r(t)| \ll \frac{C(r)}{|t|^b}, \qquad \text{for }\ |t|\ge M. \end{equation*}
Then we have
\begin{equation} \label{lem2a}
\begin{aligned}
\sum_{0<\gamma,\gamma'\leq T} r\left(\frac{\gamma-\gamma'}{2\pi}\log T\right)w(\gamma-\gamma') = \sum_{(\gamma,\gamma')\in \mathcal{P}(T,M)} &r\left(\frac{\gamma-\gamma'}{2\pi}\log T\right)w(\gamma-\gamma') \\
&\qquad+ O\left( \frac{ C(r) T\log T}{M^{b-1}}\right) + O(T).
\end{aligned}
\end{equation}
We may delete the factor $w(\gamma-\gamma')$ from the sum on the right-hand side if we wish. 
\end{lem}

Next we evaluate the sum on the right-hand side of \eqref{lem2a} with $w(\gamma-\gamma')$ removed using \hyperref[AH2]{\textbf{AH-Pairs}}. 
\begin{lem} \label{AHsumlemma} With $r(t)$ as in \Cref{sumerror}, we have assuming \hyperref[AH2]{AH-Pairs} that
\begin{equation} \label{AHrsum} \sum_{(\gamma,\gamma')\in \mathcal{P}(T,M)}r\left(\frac{\gamma-\gamma'}{2\pi}\log T\right) = \frac{T}{2\pi}\log T\sum_{\substack{k\in \mathbb{Z} \\ k\ll M}} r(k/2)P_{k/2} + O(M^2R(T)T\log T).
\end{equation}
\end{lem}

To complete the proof of \Cref{thm1} we make use of a special Fourier transform pair. 
\begin{lem} \label{lemg_n} For $n\ge 1$ a positive integer, let
\begin{equation*} g_n(\alpha) =
\begin{cases}\sin^2(\frac{n\pi}{2}\alpha),
 &|\alpha|\le1,\quad \text{ $n$ even,} \\
\cos^2(\frac{n\pi}{2}\alpha), &|\alpha|\le1, \quad \text{ $n$ odd,}\\
0, &|\alpha|>1.
\end{cases}
\end{equation*}
Clearly $g_n(\alpha)$ is even and in $ L^{\!1}(\mathbb R)$, and we have
\begin{equation*} 
\widehat{g}_n(t) = \frac{\sin(2\pi t)}{2\pi t}\left(\frac{n^2}{n^2-4t^2}\right), \end{equation*}
and thus
\begin{equation} \label{g_nbound}
|\widehat{g}_n(t)| \ll \begin{cases}
\frac{1}{|t|+1}, &\text{if }~ 0\le |t|\le n/4, \\
1, &\text{if }~ n/4<|t|<n, \\
\frac{n^2}{|t|^3} &\text{if }~ |t|\ge n.
\end{cases}
\end{equation}
Therefore $\widehat{g}_n(t) \in L^{\!1}(\mathbb R)$.
Finally, for $k\in \mathbb Z$, we have
\begin{equation}\label{hatg_nzero} \widehat{g}_n(k/2) = \begin{cases}
1, & \text{if $k=0$,}\\
\frac{(-1)^{n+1}}2, & \text{if $k= \pm n$,}\\
0 & \text{otherwise.}
\end{cases}
\end{equation}
\end{lem}
\noindent The middle bound in \eqref{g_nbound} is crude but suffices for our needs, and by \eqref{hatg_nzero} cannot be improved when $t=n/2$. 

In our proof of \Cref{thm1} we use the Fourier transform pairs from \Cref{lemg_n} in \hyperref[lemsumpair]{\textbf{MT-Pairs}}, where both $g_n(\alpha)$ and $\widehat{g}_n(t)$ are in $L^{\!1}(\mathbb R)$, $g_n(\alpha)$ has support in $|\alpha|\le 1$, and $\widehat{g}_n(t)$ vanishes for all half integers $k/2$ except the three values $k/2=0, \pm n/2$. This allows us to determine $P_0$ in terms of $P_{n/2}=P_{-n/2}$. Turning now to \Cref{thm3-simple_zeros}, this suggests that if we had a similar kernel which vanished on all half-integers except $k/2=0$ we would be able to determine $P_0$, which formally could be done by taking $n\to \infty$ for $\widehat{g}_n(t)$ in \Cref{lemg_n}. Proceeding directly, Montgomery's first example of a kernel used in \hyperref[lemsumpair]{\textbf{MT-Pairs}} was the Fourier transform pair 
\begin{equation} \label{k-kernel} \widehat{k}(t)= \frac{\sin(2\pi\beta t)}{2\pi \beta t}, \qquad k(\alpha) = \begin{cases}
\frac{1}{2\beta}, &\text{if } |\alpha|\le \beta, \\
0, &\text{if } |\alpha|> \beta.
\end{cases}
\end{equation}
From \eqref{pairsum} of \hyperref[lemsumpair]{\textbf{MT-Pairs}}, Montgomery \cite[Corollary 1]{Montgomery73} obtained, assuming RH and fixed $0<\beta \le 1$,
\begin{equation} \label{Msum1} \sum_{0<\gamma , \gamma'\le T}\left( \frac{\sin \beta(\gamma-\gamma')\log T}{\beta (\gamma-\gamma')\log T}\right) w(\gamma-\gamma') = \left(\frac1{2\beta}+ \frac{\beta}{2}\right) \frac{T}{2\pi} \log T + O\left(T\sqrt{\log T}\right). \end{equation}
While the kernels in \eqref{k-kernel} have the desired properties, unfortunately $\widehat{k}(t)\not \in L^{\!1}(\mathbb R)$ and the sum on the left-hand side of \eqref{Msum1} is only conditionally convergent without the weight $w(\gamma-\gamma')$. This is why all applications of \hyperref[lemsumpair]{\textbf{MT-Pairs}} use Fourier transform pairs $g$ and $\widehat{g}$ where both $g_n$ and $\widehat{g}$ are in $L^{\!1}(\mathbb R)$. However, in proving \Cref{thm3-simple_zeros} we have available \hyperref[AH2Strong]{\textbf{Strong AH-Pairs}} which handle the difference $\gamma-\gamma'$ up to bounded size, at which point the weight $w(\gamma-\gamma')$ damps out the remaining part of the sum. 

We next turn to determining the shape of $F(\alpha)$ as illustrated in \Cref{sec:1}, where we assume RH and \hyperref[AH2]{\textbf{AH-Pairs}} and do not restrict ourselves to $|\alpha|\le 1$.
Our results follow the same method used in \cite{Bal16} where the assumption that $P_0=p_0=1$ was applied. The proof there can be easily modified to also handle the case where $\limsup P_0>1$.

We now define the delta function $\delta_a(x)$ intuitively by \[ \delta_a(x)=0 \quad \text{if} \quad x\neq a, \qquad \int_{a-\eta}^{a+\eta}\delta_a(x)\, dx = 1 \quad \text{for any $\eta>0$}. \]
We may also have additional one-sided information on a delta function; here we consider three types of such delta functions $\delta_a^=(x)$, $\delta_a^+(x)$, $\delta_a^-(x)$, where for any $\eta>0$,
\[ \int_{a-\eta}^{a}\delta_a^=(x)\, dx =\int_{a}^{a+\eta}\delta_a^=(x)\, dx = 1/2 , \]
\[ \int_{a-\eta}^{a}\delta_a^+(x)\, dx = 0, \qquad \int_{a}^{a+\eta}\delta_a^+(x)\, dx = 1,\]
\[ \int_{a-\eta}^{a}\delta_a^-(x)\, dx = 1, \qquad \int_{a}^{a+\eta}\delta_a^-(x)\, dx = 0.\]
Define
\begin{equation}\label{calF}
\begin{aligned}
\mathcal{F}(\alpha) &:= s(\alpha) + \delta_0^=(\alpha) +
\delta_1^+(\alpha) + \delta_{-1}^-(\alpha) \\
&\qquad+ \sum_{n=1}^\infty (\delta_{2n}(\alpha)+\delta_{-2n}(\alpha)) + 2(P_0-1)\sum_{n=2}^\infty \left(\delta_{2n-1}(\alpha) + \delta_{1-2n}(\alpha)\right),
\end{aligned}
\end{equation}
where 
\begin{equation} \label{s(alpha)}
s(\alpha)=|\alpha| \quad\text{for } |\alpha|\le 1, \qquad \text{and} \qquad
s(\alpha+2) = s(\alpha) \quad\text{for all $\alpha$}. \end{equation}
If we only make use of the ordinary delta function then \eqref{calF} is equivalent to
\begin{equation} \label{Fin[0,2]}
\mathcal{F}(\alpha) = \min(\alpha, 2-\alpha) + \delta_{0}(\alpha) + 2(P_0-1)\delta_{1}(\alpha)\qquad \text{for} \qquad 0\le \alpha \le 2,
\end{equation}
together with the properties 
\begin{equation} \label{evenperiod2}
\mathcal{F}(-\alpha)=\mathcal{F}(\alpha), \quad \mathcal{F}(\alpha+2)=\mathcal{F}(\alpha), \qquad \text{for} \quad \alpha\ge 0.
\end{equation}
\begin{theorem} \label{thm4-r(k)p_k} Let $g$ be a Riemann integrable function on a closed interval $[a,b]$, where $a,b \not \in \mathbb{Z}\setminus \{-1,0,1\}$. Suppose in addition that $g$ is Lipschitz continuous at $\mathbb{Z}\setminus \{-1,1\}$ and also right-Lipschitz continuous at $1$ and left-Lipschitz continuous at $-1$. Assuming the Riemann Hypothesis and \hyperref[AH2]{AH-Pairs}, we have as $T\to \infty$ that
\begin{equation} \label{F-calF}
\int_{a}^b F(\alpha)g(\alpha)\, d\alpha = \int_a^b \mathcal{F}(\alpha)g(\alpha)\, d\alpha +o(1), \end{equation}
where $\mathcal{F}$ is given in \eqref{calF} and is the bottom plot in Figure \ref{fig-AH-triangles}.  
\end{theorem}

\begin{remark} Because of \hyperref[MonThm]{\textbf{MT}}, we are able to determine on RH the one-sided delta functions at $\alpha=-1,0,1$. Beyond this range we use \hyperref[AH]{\textbf{AH}}, but our method only allows one to determine averages of $F$ on short intervals around a given point, and thus we cannot determine the one-sided delta functions at any other integers. Since in applications we often want to take $a$ or $b$ to be selected from $\{-1,0,1\}$, we have gone to the trouble to include the limited information we have on one-sided delta functions in \Cref{thm4-r(k)p_k}. 
\end{remark} 

It is often easy to do calculations using \Cref{thm4-r(k)p_k}. One example is the following result concerning a constant\footnote[3]{David Farmer suggested to us that we compute this constant.} that has occurred in an application of pair correlation to the second moment of $S(T)$ \cite{Goldston87}.

\begin{corollary} \label{cor4}
Let 
\begin{equation*} \mathcal{C} := \lim_{T\to\infty} \int_1^\infty \frac{F(\alpha)}{\alpha^2}\, d\alpha \end{equation*}
if this limit exists. Then the following statements hold.
\begin{enumerate}[label={(\alph*)}]
\item Assuming the Riemann Hypothesis and Montgomery's conjecture that $F(\alpha)\sim 1$ for $1\le \alpha \le U$ for arbitrarily large $U$. Then $\mathcal{C}=1$.
\item Suppose that the limiting density $p_0$ exists, then assuming the Riemann Hypothesis and \hyperref[AH2]{AH-Pairs} we obtain
$$ \mathcal{C}= 1 + \left(\frac32(p_0-1)+\frac14\right)\frac{\pi^2}6 + \log\frac2\pi. $$
\begin{itemize}
\item In the particular case when $p_0=1$, the Riemann Hypothesis and \hyperref[AH2]{AH-Pairs} imply
 $$ \mathcal{C}= 1 + \frac{\pi^2}{24} + \log\frac2\pi = 0.95965\dots . $$
\item In the other extreme case $p_0=\frac32-\frac{2}{\pi^2}$, the Riemann Hypothesis and \hyperref[AH2]{AH-Pairs} then imply
 $$ \mathcal{C}= \frac12 + \frac{\pi^2}6 + \log\frac2\pi=1.69335\ldots . $$
\end{itemize}
\end{enumerate}
\end{corollary}

We conclude with an application of \hyperref[AH-density]{\textbf{AH-Density}}. 

\begin{theorem} \label{thm2-generalizedMTwith_p_k} Suppose both $r(t)$ and $\widehat{r}(\alpha)$ are even $L^1(\mathbb{R})$ functions with bounded variation, and suppose further that $r(\alpha)$ has support in $|\alpha|\le 1$. Then assuming \hyperref[AH-density]{AH-Density}, we have
\begin{equation} \sum_{k\in \mathbb{Z}} \widehat{r}(k/2)p_{k/2} = r(0) + 2 \int_0^1\alpha r(\alpha)\, d\alpha. \end{equation}
\end{theorem} 

Combining this result with \Cref{AHsumlemma} we thus obtain a version of \hyperref[lemsumpair]{\textbf{MT-Pairs}} making use of \hyperref[AH2]{\textbf{AH-Pairs}} and \hyperref[AH-density]{\textbf{AH-Density}} but not RH or explicit formulas. One could view this as using a false assumption to prove a true theorem, but it also indicates that the conditions on \hyperref[AH-density]{\textbf{AH-Density}} that were imposed by \hyperref[lemsumpair]{\textbf{MT-Pairs}} and \hyperref[AH2]{\textbf{AH-Pairs}} using the special kernel in \Cref{lemg_n} are very close to optimal over a wide range of kernels.

\section{Proof of \Cref{thm1}}

\begin{proof}[Proof of {\hyperref[lemsumpair]{\textbf{MT-Pairs}}}] To prove \eqref{Sumpair}, letting 
$t= \frac{\gamma-\gamma'}{2\pi}\log T$ in \eqref{hat-g}, we have
\[
\widehat{g}\Big(\frac{\gamma-\gamma'}{2\pi}\log T\Big) = \int_{-\infty}^\infty g(\alpha)T^{i\alpha(\gamma'-\gamma)}\,d\alpha\,.
\]
Multiplying both sides of this equation by $w(\gamma-\gamma')$ and summing over $0<\gamma, \gamma'\le T$ we obtain
\begin{align*}
\sum_{0<\gamma,\gamma'\le T}\widehat{g}\Big(\frac{\gamma-\gamma'}{2\pi}\log T\Big)w(\gamma-\gamma')
&= \int_{-\infty}^\infty g(\alpha)\sum_{0<\gamma,\gamma' \le T } T^{i\alpha(\gamma'-\gamma)}w(\gamma-\gamma')\,d\alpha \notag \\
&= \left(\frac{T}{2\pi}\log T\right) \int_{-\infty}^\infty F(\alpha)g(\alpha)\,d\alpha ,
\end{align*}
since $w$ is even.
To obtain \eqref{pairsum}, we see by \hyperref[MonThm]{\textbf{MT}} that
\begin{align*}
\int_{-\infty}^\infty F(\alpha)r(\alpha)\,d\alpha &= 2\int_0^1 T^{-2\alpha}\log T\left(1+ O\left(\frac{1}{\sqrt{\log T}}\right)\right) g(\alpha)\, d\alpha \\
&\qquad+ 2\int_0^1\alpha g(\alpha)\, d\alpha +O\left(\frac{1}{\sqrt{\log T}}\right).
\end{align*}
Using the Lipschitz condition on $g(\alpha)$ at $\alpha=0$, we see that the first integral on the right-hand side is 
\[ \begin{split} &= 2\int_0^{\log\log T/\log T} T^{-2\alpha}(\log T +\sqrt{\log T}) g(\alpha) \, d\alpha+O\left( \int_{\log\log T/\log T}^1 T^{-2\alpha}\log T|g(\alpha)|\, d\alpha \right) \\&
= 2\left( g(0)+ O\left(\frac{\log\log T}{\log T}\right)\right) \int_0^{\log\log T/\log T}T^{-2\alpha}(\log T +\sqrt{\log T}) \, d\alpha \\
&\qquad+ O\left(\frac1{\log T}\int_0^1|g(\alpha )|\, d\alpha\right)\\&
= 2\left( g(0)+ O\left(\frac{\log\log T}{\log T}\right)\right)\left(\frac12 +O\left(\frac1{\sqrt{\log T}}\right)\right) + O\left(\frac1{\log T}\right) \\&
= g(0) + O\left(\frac{1}{\sqrt{\log T}}\right).
\end{split}\]
\end{proof}

To prove \Cref{sumerror}, we need some estimates for zeros. Note first from \eqref{N(T)} that 
\begin{equation*} 
N(t+1)-N(t) = \sum_{ t<\gamma \le t+1} 1\ll \log( |t|+2) , \end{equation*}
which easily implies (or see \cite[Chapter 15 p. 98]{Dav2000})
\begin{equation} \label{wSum} \sum_\gamma \frac{1}{1 + (t-\gamma)^2 } \ll \log (|t|+2) .\end{equation}
We will also need the following result from \cite[Lemma 9]{GM87}. 
If $0\le h\le T$, then 
\begin{equation} \label{zeropairbound}
|\{(\gamma,\gamma') : 0<\gamma\le T, |\gamma-\gamma'|\le h\}| = \sum_{\gamma'}\sum_{\substack{ 0< \gamma \le T, \\ |\gamma-\gamma'|\le h}} 1 \ll (1+h\log T) T\log T.
\end{equation}
One immediate consequence of the above estimate is that recalling $\mathcal{P}(T,M)$ from \eqref{P(T,M)}, we have on taking $h=M/\log T$ in \eqref{zeropairbound} that
\begin{equation} \label{P(T,M)bound}
\mathcal{P}(T,M)\ll MT\log T. \end{equation}
\begin{proof}[Proof of \Cref{sumerror}]
We first note that, on using \eqref{N(T)}, \eqref{wSum}, and that $r(t)\ll 1$,
\begin{equation} \label{GetP(T,M)}
\begin{aligned}
&\sum_{0<\gamma,\gamma'\leq T} r\left(\frac{\gamma-\gamma'}{2\pi}\log T\right)w(\gamma-\gamma') \\ 
&\quad= \left(\sum_{\frac{T}{\log^2T}<\gamma,\gamma'\leq T} +
\sum_{0<\gamma'\leq\frac{T}{\log^2T}} \sum_{0<\gamma\leq T} +
\sum_{\frac{T}{\log^2T}<\gamma'\leq T}\sum_{0<\gamma\leq\frac{T}{\log^2T}} \right) r\left(\frac{\gamma-\gamma'}{2\pi}\log T\right)w(\gamma-\gamma') \\
&\quad= \sum_{\frac{T}{\log^2T}<\gamma,\gamma'\leq T} r\left(\frac{\gamma-\gamma'}{2\pi}\log T\right)w(\gamma-\gamma') + O\left(\sum_{0<\gamma'\leq\frac{T}{\log^2T}} \sum_{0<\gamma\leq T} \frac1{4+(\gamma-\gamma')^2} \right) \\
&\quad= \sum_{\frac{T}{\log^2T}<\gamma,\gamma'\leq T} r\left(\frac{\gamma-\gamma'}{2\pi}\log T\right)w(\gamma-\gamma') + O\left(\sum_{0<\gamma'\leq\frac{T}{\log^2T}} \log T \right) \\
&\quad= \sum_{\frac{T}{\log^2T}<\gamma,\gamma'\leq T} r\left(\frac{\gamma-\gamma'}{2\pi}\log T\right)w(\gamma-\gamma') + O\left(T\right).
\end{aligned}
\end{equation}
The sum here is over $(\gamma,\gamma')\in\mathcal{P}(T,M)$ plus terms with $|(\gamma-\gamma')\log T/2\pi| >M$, which we may remove using \eqref{zeropairbound} with an error
\begin{align*}
&\ll \sum_{\substack{0<\gamma,\gamma'\leq T \\ M<\left| \frac{\gamma-\gamma'}{2\pi}\log T\right|}} \left| r\left( \frac{\gamma-\gamma'}{2\pi}\log T\right)\right| \\ &
\ll \sum_{\ell=0}^{\infty} \sum_{\substack{0<\gamma,\gamma'\leq T \\ 2^{\ell} M<\left| \frac{\gamma-\gamma'}{2\pi}\log T\right| \leq 2^{\ell+1}M }} \left| r\left( \frac{\gamma-\gamma'}{2\pi}\log T\right)\right| \\
& \ll \sum_{\ell=0}^{\infty} \frac{C(r)}{(2^{\ell}M)^b}\sum_{\substack{0<\gamma,\gamma'\leq T \\ \left| \gamma-\gamma'\right| \leq 2^{\ell+2}\pi M/\log T}} 1 \\
& \ll \sum_{\ell=0}^{\infty}\frac{C(r)}{(2^{\ell}M)^b}( 2^{\ell}M T\log T) \\&
\ll \frac{ C(r) T\log T}{M^{b-1}}.
\end{align*}
Thus we have proved \eqref{lem2a}.
By \eqref{P(T,M)bound} we can now remove the factor $w(\gamma-\gamma')$ in the sum on the right-hand side of \eqref{lem2a} with an error 
\[ \ll \max_{|(\gamma-\gamma')| \le 2\pi M/\log T}\left|\frac{4}{4+ (\gamma-\gamma')^2} -1\right||\mathcal{P}(T,M)| \ll \left(\frac{M}{\log T}\right)^2 MT\log T \ll \frac{M^3 T}{\log T} \ll T,\]
which proves the last statement in \Cref{sumerror}.
\end{proof}

\begin{proof}[Proof of \Cref{AHsumlemma}]
With $r(t)$ as in \hyperref[lemsumpair]{\textbf{MT-Pairs}} and \Cref{sumerror}, we have first
\begin{equation*}\begin{split}
|r(t+h)-r(t) |
&= \Bigg| \int_{-1}^{1} \hat{r}(\alpha) \Big(e((t+h)\alpha)-e(t\alpha)\Big)\,d\alpha\Bigg| \\
&\leq \max_{-1\leq\alpha\leq 1}\Big| e(h\alpha)-1 \Big| \int_{-1}^{1} |\hat{r}(\alpha)|\,d\alpha \\
&\ll \min\{1,|h|\}.
\end{split}\end{equation*}
Hence by \hyperref[AH2]{\textbf{AH-Pairs}}, for $(\gamma,\gamma')\in B_{k/2} \subset \mathcal{P}(T,M)$, 
\[ r\left(\frac{\gamma-\gamma'}{2\pi}\log T\right) = r\left(\frac{k}{2} + O((|k|+1)R(T))\right) = r(k/2) + O((|k|+1)R(T)).\]
Thus, the left-hand side of \eqref{AHrsum} is by \eqref{P(T,M)bound} equal to 
\[ \sum_{\substack{k\in \mathbb{Z} \\ |k|\ll M}} \Big( r(k/2) + O((|k|+1)R(T))\Big)|B_{k/2}| = \frac{T}{2\pi}\log T\sum_{\substack{k\in \mathbb{Z} \\ |k|\ll M}} r(k/2)P_{k/2} + O(M^2R(T)T\log T),\]
since 
\[ \sum_{\substack{k\in \mathbb{Z} \\ |k|\ll M}} B_{k/2} \ll \mathcal{P}(T,M)\ll M T\log T \]
by \eqref{P(T,M)bound}.
\end{proof}

\begin{proof}[Proof of Lemma \ref{lemg_n}] Since
\begin{align} \label{g_n}
g_n(\alpha) =
\begin{cases}
\displaystyle \frac{1+(-1)^{n+1}\cos(n\pi\alpha)}2, &|\alpha|\le1, \\[2mm]
0, &|\alpha|>1,
\end{cases}
\end{align}
we have by \eqref{hat-g}
\begin{align*}
\widehat{g}_n(t) 
&= \int_0^1 \left(1+(-1)^{n+1}\cos(n\pi\alpha)\right)) \cos(2\pi t\alpha) d\alpha \\
&= \frac{\sin(2\pi t)}{2\pi t} + \frac{(-1)^{n+1}}2 \int_0^1 \Big(\cos(\pi\alpha(n+2t))+\cos(\pi\alpha(n-2t))\Big) d\alpha \\
&= \frac{\sin(2\pi t)}{2\pi t} + \frac{(-1)^{n+1}}2 \left(\frac{\sin\pi(n+2u)}{\pi(n+2t)}+\frac{\sin\pi(n-2t)}{\pi(n-2t)}\right) \\
&= \frac{\sin(2\pi t)}{2\pi t} + \frac{(-1)^{n+1}}{2\pi} \left(\frac{(-1)^n\sin 2\pi t}{n+2t}-\frac{(-1)^n\sin 2\pi t}{n-2t}\right) \\
&= \frac{\sin(2\pi t)}{2\pi t} - \frac{\sin 2\pi t}{2\pi} \left(\frac1{n+2t}-\frac1{n-2t}\right) \\
&= \frac{\sin(2\pi t)}{2\pi t} \left( 1 + \frac{4t^2}{n^2-4t^2}\right) \\
&= \frac{\sin(2\pi t)}{2\pi t} \left(\frac{n^2}{n^2-4t^2}\right).
\end{align*}
Clearly both $g_n(t)$ and $\widehat{g}_n(\alpha)$ are even and $\widehat{g}_n(\alpha) \in L^{\!1}(\mathbb R)$.

We now prove \eqref{g_nbound}. Making use of $|(\sin x)/x|\le \min(1, 1/|x|)$, we see that for $0\le |t| \le n/4$ 
\[ |\widehat{g}_n(t)| \le \frac43 \min(1, 1/t) \ll \frac{1}{t +1}, \]
and for $|t|\ge n$ that 
\[|\widehat{g}_n(t)| \leq \frac{n^2}{6\pi t^3}.\]
Finally for $n/4< |t| <n$ we have 
\[ |\widehat{g}_n(t)| = \left|\frac{\sin\pi(2t-n)}{\pi(2t-n)}\right|\left| \frac{n^2}{2 t (2t+n)}\right| \le \frac{4}{3}.\]

To prove \eqref{hatg_nzero}, for $k\in \mathbb Z$, we clearly have $\widehat{g}_n(0)=1$ and $\widehat{g}_n(k/2) =0$ if $k\neq \pm n$. For $k=\pm n$, we have by L'H\^opital's rule 
\[ \widehat{g}_n(\pm n/2) = \pm \frac{n}{ \pi } \ \lim_{t\to \pm n/2} \frac{\sin( 2\pi t)}{n^2-4t^2} = \frac{(-1)^{n+1}}{2}.\]
\end{proof}
\begin{proof}[Proof of \Cref{thm1}] In \Cref{sumerror} and \Cref{AHsumlemma} we take $r(t) = \widehat{g}_n(t)$ from Lemma \ref{lemg_n}. We take $M$ large enough so that $n\ll M$ and note from \eqref{g_nbound} that we can take $C(r)=n^2$ and $b=3$. Thus, 
\[\begin{split}
&\sum_{0<\gamma,\gamma'\leq T} \widehat{g}_n\left(\frac{\gamma-\gamma'}{2\pi}\log T\right)w(\gamma-\gamma') \\
&\qquad= \left(\sum_{\substack{k\in \mathbb{Z} \\ |k|\ll M}} \widehat{g}_n(k/2)P_{k/2} + O\left(\frac{n^2}{M^2}\right)+ O\left(M^2R(T)\right) + O\left(\frac{1}{\log T}\right)\right) \frac{T}{2\pi}\log T \\
&\qquad= \left(P_0+ (-1)^{n+1}P_{n/2} + O\left(\frac{n^2}{M^2}\right)+O\left(M^2R(T)\right)+ O\left(\frac{1}{\log T}\right)\right) \frac{T}{2\pi}\log T.
\end{split}\]
By \hyperref[lemsumpair]{\textbf{MT-Pairs}} we have the sum on the left above is 
\begin{align*}
&= (1+o(1))\frac{T}{2\pi}\log T \left( 
g_n(0) + 2\int_0^1\alpha g_n(\alpha)\,d\alpha \right) \\
&= (1+o(1))\frac{T}{2\pi}\log T
\begin{cases}
\displaystyle \frac12, &\text{for $n$ even}, \\[2mm]
\displaystyle \frac32 -\frac{2}{\pi^2n^2}, &\text{for $n$ odd},
\end{cases}
\end{align*}
by an elementary calculation using \eqref{g_n}. 
We now take $T$ and then $M$ large in the last two equations and conclude 
\[ P_0+ (-1)^{n+1}P_{n/2} \sim
\begin{cases}
\displaystyle \frac12, &\text{when $n$ is even}, \\[2mm]
\displaystyle \frac32-\frac{2}{\pi^2n^2}, &\text{when $n$ is odd},
\end{cases} \]
proving \eqref{thm1b}. For \eqref{thm1a}, note when $n=1$ that $P_0+P_{1/2} \sim \frac32 -\frac{2}{\pi^2}$, and that $P_0\ge (\frac{T}{2\pi}\log T)^{-1}N(T) \ge 1+o(1)$ and $P_{1/2}\ge 0$. Thus
\[1+o(1)\le P_0 \le \frac32 -\frac{2}{\pi^2} -P_{1/2} + o(1)\le\frac32 -\frac{2}{\pi^2} +o(1), \] 
which completes the proof.
\end{proof}

\section{Proof of \Cref{thm3-simple_zeros}}

\begin{proof}[Proof of \Cref{thm3-simple_zeros}] Assuming RH, we obtain from \eqref{Msum1} with $\beta=1$ that
\begin{equation*} \sum_{0<\gamma,\gamma'\le T}\left(\frac{\sin\left((\gamma-\gamma')\log T\right)}{(\gamma-\gamma')\log T}\right) w(\gamma-\gamma') = \frac{T}{2\pi}\log T + O(T\sqrt{\log T}). \end{equation*}
By \eqref{GetP(T,M)}, we have, since $|\frac{\sin t}{t}|\le 1 $,
\[ \sum_{\frac{T}{\log^2 T} <\gamma,\gamma' \le T}\left(\frac{\sin\left((\gamma-\gamma')\log T\right)}{(\gamma-\gamma')\log T}\right) w(\gamma-\gamma') = \frac{T}{2\pi} \log T +O(T\sqrt{\log T}).\]
We now discard the terms with $|\gamma-\gamma'| > \mathcal{M}$ which by \eqref{zeropairbound} contributes an error 
\[ \begin{split} &\ll \sum_{\ell=0}^\infty \sum_{\substack{0 <\gamma,\gamma' \le T\\ 2^\ell\mathcal{M} <|\gamma-\gamma'|\le 2^{\ell+1} \mathcal{M}}} \frac{1}{(\gamma-\gamma')^3\log T}\\&
\ll \sum_{\ell=0}^\infty \frac{1}{2^{3\ell}\mathcal{M}^3\log T} \sum_{\substack{0 <\gamma,\gamma' \le T\\ |\gamma-\gamma'|\le 2^{\ell+1} \mathcal{M}}} 1 \\&
\ll \sum_{\ell=0}^\infty \frac{2^\ell\mathcal{M}T\log^2T}{2^{3\ell}\mathcal{M}^3\log T} \\&
\ll \frac{T\log T}{\mathcal{M}^2}.
\end{split}\]
Hence we have
\begin{equation} \label{ProofThm2eq}
\begin{aligned}
&\sum_{(\gamma,\gamma')\in \mathcal{Q}(T,\mathcal{M})}\left(\frac{\sin\left((\gamma-\gamma')\log T\right)}{(\gamma-\gamma')\log T}\right) w(\gamma-\gamma') \\
&\qquad= \frac{T}{2\pi} \log T + O(T\sqrt{\log T}) + O\left(\frac{T\log T}{\mathcal{M}^2}\right).
\end{aligned}
\end{equation}
By \hyperref[AH2Strong]{\textbf{Strong AH-Pairs}} and the comment on how $B_{k/2}$ in \eqref{Bk/2} is defined for $\mathcal{Q}(T,\mathcal{M})$, the sum above is
\begin{align*}
&= \sum_{(\gamma,\gamma')\in B_0}\left(\frac{\sin\left((\gamma-\gamma')\log T\right)}{(\gamma-\gamma')\log T}\right) w(\gamma-\gamma') \\
&\qquad+ O\left(\sum_{\substack{k\in \mathbb{Z} \\ 0< |k|\ll \mathcal{M}\log T}} \left(\frac{|\sin (\pi k + O((|k|+1)R(T))|}{\pi k + O((|k|+1)R(T))}\right) |B_{k/2}|\right).
\end{align*}
Since $(\sin x)/x = 1+O(x^2)$ if $x\to 0$ and $w(u) = 1 + O(u^2)$ if $u\to 0$, the first term above is 
\[ = (1+ O(R(T)^2))(1+ O(R(T)^2/\log^2T))|B_0| = (1+ O(R(T)^2))P_0\frac{T}{2\pi}\log T .\]
Using $|\sin(\pi k + x)|= |\sin x|\le |x|$, we bound the second term above by
\[ \ll R(T)|\mathcal{Q}(T,\mathcal{M})| \ll \mathcal{M}R(T)T\log^2T\]
where we used \eqref{zeropairbound}. Hence \eqref{ProofThm2eq} becomes 
\[ P_0\frac{T}{2\pi}\log T + O(\mathcal{M}R(T)T\log^2T) = \frac{T}{2\pi} \log T + O(T\sqrt{\log T}) + O\left(\frac{T\log T}{\mathcal{M}^2}\right)\] 
which implies
\[ P_0 = 1 + O\left(\frac{1}{\mathcal{M}^2}\right) + O\left(\frac1{\sqrt{\log T}}\right) + O(\mathcal{M}R(T)\log T). \]
Letting $T\to \infty$, we have $R(T)\log T \to 0$, and on taking $\mathcal{M}$ large we have
$p_0 = \lim_{T\to \infty}P_0 =1$.
\end{proof}

\section{Proof of a Lemma of Heath-Brown}

Following Heath-Brown \cite{Heath-Brown96}, define
for $\alpha\in\mathbb{R}$ and $\lambda>0$ 
\begin{equation} \label{G_lambda} G_\lambda(\alpha) := \frac{1}{\lambda^2} \int_{-\lambda}^{\lambda} \left(\lambda -|\beta| \right) F(\alpha+\beta)\, d\beta = \frac{1}{\lambda^2}\int_{-\infty}^\infty k_\lambda(\beta)F(\alpha+\beta)\, d\beta,
\end{equation}
where $k_\lambda(\beta)$ is the triangle function
\begin{equation} \label{triangle}
k_\lambda(\beta) := \max(\lambda -|\beta|,0).
\end{equation}
Making use of the Fej\'er kernel
\begin{equation} \label{Fejer} K_\lambda(y) := \left( \frac{\sin \frac{\lambda y}2}{\frac{\lambda y}2}\right)^2 = \frac{1}{\lambda^2}\int_{-\infty}^\infty k_\lambda(\beta) e^{iy\beta}\, d\beta , \end{equation}
we have from \eqref{G_lambda} that
\begin{equation} \label{G_lambda2}
G_\lambda(\alpha) = \left(\frac {T}{2\pi} \log T\right)^{-1} \sum_{0<\gamma, \gamma'\le T}T^{i\alpha(\gamma-\gamma')} K_\lambda((\gamma-\gamma')\log T) w(\gamma-\gamma').
\end{equation} 
To state our first results on $G_\lambda(\alpha)$, it is convenient to condense our main error terms in applying \hyperref[AH2]{\textbf{AH-Pairs}} into the form
\begin{equation} \label{E_G} E_G(\lambda, \alpha) := \frac1{\lambda^2 M} + (|\alpha|+1)M^2R(T) +\frac1{\log T},\end{equation}
where $M$, $T$, and $R(T)$ are from \hyperref[AH2]{\textbf{AH-Pairs}}.
\begin{lem}[Heath-Brown] \label{lem4-Gfunction} For $\alpha \in \mathbb{R}$ and $0<\lambda<1$ the function $G_\lambda(\alpha)$ is even and non-negative. Let $\lambda$ satisfy $0< \lambda \le 1/2$. Assuming the Riemann Hypothesis, we have
\begin{enumerate}[label={(\roman*)}]
\item \label{lem4-i} $\displaystyle G_\lambda(\alpha) = \frac{\lambda-|\alpha|}{\lambda^2} + O(\lambda) + O\left(\frac{1}{\lambda\sqrt{\log T}}\right)+O\left(\frac{1}{\lambda^2\log T}\right), \qquad \text{for}\qquad |\alpha| < \lambda$, \\
\item \label{lem4-ii} $\displaystyle G_\lambda(\alpha) = |\alpha| + O\left(\frac{1}{\sqrt{\log T}}\right) + O\left(\frac{1}{\lambda^2\log T}\right), \qquad \text{for} \qquad \lambda \le |\alpha|\le 1-\lambda$.
\end{enumerate}
Assuming \hyperref[AH2]{\textbf{AH-Pairs}}, we have
\begin{enumerate}[label={(\roman*)}]
\setcounter{enumi}{2}
\item \label{lem4-iii} $\displaystyle G_\lambda(\alpha) = \sum_{k\in \mathbb{Z}} e^{i\pi k\alpha}\left( \frac{\sin \frac{\lambda\pi k}{2}}{\frac{\lambda\pi k}{2}}\right)^2P_{k/2}+ O(E_G(\lambda, \alpha))$,
\end{enumerate}
and for $L\in \mathbb{Z}$,
\begin{enumerate}[label={(\roman*)}]
\setcounter{enumi}{3}
\item \label{lem4-iv} $\displaystyle G_\lambda(\alpha+2L) = G_{\lambda}(\alpha) + O(E_G(\lambda,|\alpha|+2L)).$
\end{enumerate}
\end{lem}

Assuming RH and \hyperref[AH2]{\textbf{AH-Pairs}}, \Cref{lem4-Gfunction} determines $G_\lambda(\alpha)$ for all $\alpha$ except when $\alpha$ is an odd integer. We will determine $G_{\lambda}(2L+1)$ asymptotically in \Cref{lem5-forThm4} \ref{lem5-v}, but it is easy to prove here as a first step the following average result. 
\begin{corollary} \label{Cor5-G(1)ave} Assuming the Riemann Hypothesis and \hyperref[AH2]{AH-Pairs}, we have, for $0<\lambda \le 1/4$,
\begin{equation} \label{G(1)ave} \int_{1-\lambda}^{1+\lambda}G_{\lambda}(\alpha) \, d\alpha = 2(P_0-1) + O(\lambda) + O( E_G(\lambda,1)) +O\left(\frac{1}{\lambda^2\sqrt{\log T}}\right).
\end{equation} 
\end{corollary} 

\begin{proof}[Proof of \Cref{lem4-Gfunction}] $G_\lambda(\alpha)$ is even since in \eqref{G_lambda2} $\gamma$ and $\gamma'$ can be interchanged in the sum, and $G_\lambda(\alpha)\ge 0$ on using $F \ge 0$ in \eqref{G_lambda}. 

For \ref{lem4-i}, by the evenness of $G_\lambda(\alpha)$, we only need to consider the range $0\le \alpha <\lambda\le 1/2$. Then by \hyperref[MonThm]{\textbf{MT}} we have from \eqref{G_lambda} that
\begin{equation} \label{lem4proof}
\begin{aligned}
G_\lambda(\alpha) &= \left(1+O\left(\frac{1}{\sqrt{\log T}}\right)\right) \\
&\qquad\times\left(\frac{1}{\lambda^2} \int_{-\lambda}^{\lambda} (\lambda-|\beta|) T^{-2|\alpha +\beta|}\log T\, d\beta + \frac{1}{\lambda^2} \int_{-\lambda}^{\lambda} (\lambda-|\beta|)|\alpha+\beta| \, d\beta\right)
+ O\left(\frac{1}{\sqrt{\log T}}\right).
\end{aligned}
\end{equation}
Note since $|\beta|\le \lambda$, we have that $|\alpha +\beta |\le 2\lambda \le 1$ as required in \hyperref[MonThm]{\textbf{MT}}.
The first integral in \eqref{lem4proof} is
\[= \frac{\log T}{\lambda^2} \left( \int_{-\lambda}^{-\alpha} (\lambda+\beta) T^{2(\alpha+\beta)}\, d\beta + \int_{-\alpha}^{0} (\lambda+\beta) T^{-2(\alpha +\beta)}\, d\beta + \int_{0}^{\lambda} (\lambda -\beta) ) T^{-2(\alpha +\beta)}\, d\beta \right) .
\]
The integrals here are all elementary and we obtain that this expression is
\[ = \frac{\log T}{\lambda^2} \left(\frac{\lambda-\alpha}{\log T} + \frac{T^{-2(\lambda-\alpha)} - 2T^{-2\alpha} + T^{-2(\lambda+\alpha)}}{4\log ^2T}\right)
= \frac{\lambda-\alpha}{\lambda^2} + O\left(\frac{1}{\lambda^2\log T}\right). \]

The second integral in \eqref{lem4proof} is
\[ \ll \frac{1}{\lambda^2} \int_{-\lambda}^\lambda \lambda^2\, d\beta \ll \lambda, \]
and substituting these results into \eqref{lem4proof} proves \ref{lem4-i}. 
 
For \ref{lem4-ii}, by evenness we only need to consider the range $\lambda \le \alpha\le 1-\lambda$. By \hyperref[MonThm]{\textbf{MT}} and \eqref{G_lambda}
\[ 
G_\lambda(\alpha) = \left(1+O\left(\frac{1}{\sqrt{\log T}}\right)\right)\left(\frac{1}{\lambda^2}\int_{-\lambda}^{\lambda} (\lambda -|\beta| ) \left(T^{-2(\alpha +\beta)}\log T + (\alpha +\beta)\right) \, d\beta\right)+O\left(\frac{1}{\sqrt{\log T}}\right), \]
where we note \hyperref[MonThm]{\textbf{MT}} applies here because $0\le \alpha +\beta \le 1$.
Since
\[\begin{split}
\frac{1}{\lambda^2}\int_{-\lambda}^{\lambda} (\lambda-|\beta|)T^{-2(\alpha+\beta)}\log T \,d\beta
&= \frac{T^{-2\alpha}\log T}{\lambda^2}\left(\int_{-\lambda}^0 (\lambda+\beta)T^{-2\beta} \,d\beta + \int_0^{\lambda} (\lambda-\beta)T^{-2\beta} \,d\beta\right) \\
&= \frac{T^{-2(\alpha-\lambda)} - 2T^{-2\alpha} + T^{-2( \alpha+\lambda)}}{4\lambda^2\log T}
= O\left(\frac{1}{\lambda^2\log T}\right),
\end{split} \] 
and
\[\frac{1}{\lambda^2} \int_{-\lambda}^{\lambda} (\lambda-|\beta|)(\alpha+\beta) \,d\beta
= \alpha \left(\frac{2}{\lambda^2} \int_0^{\lambda} (\lambda-\beta) \,d\beta\right) = \alpha, \] 
we obtain \ref{lem4-ii}.

For \ref{lem4-iii}, we note $K_\lambda(y) \ll \min(1, 1/(\lambda y)^2)$ and apply \Cref{sumerror} with $r(y/2\pi) = e^{i\alpha y}K_\lambda(y)$, $b=2$ and $C(K_\lambda)= 1/\lambda^2$. Thus by \eqref{G_lambda2}
\[ G_\lambda(\alpha) = \left(\frac {T}{2\pi} \log T\right)^{-1}\sum_{(\gamma,\gamma')\in \mathcal{P}(T,M)} T^{i\alpha (\gamma-\gamma')}K_\lambda((\gamma-\gamma')\log T) + O\left(\frac{1}{\lambda^2 M}\right) + O\left(\frac{1}{\log T}\right).\]
Now as in the proof of \Cref{AHsumlemma}, 
\[\begin{split} \left|K_\lambda(y+h) -K_\lambda(y)\right| &= \left| \frac{1}{\lambda^2}\int_{-\lambda}^{\lambda} \left(\lambda -|\beta| \right) \left(e^{i(y+h)\beta}-e^{iy\beta}\right)\, d\beta \right| \\
&\le \max_{|\beta|\le \lambda}\left|e^{ih\beta}-1\right| \frac{1}{\lambda^2}\int_{-\lambda}^{\lambda} \left(\lambda-|\beta|\right)\, d\beta \\
&\ll \min\{1,\lambda h\}. \end{split}\]
By \hyperref[AH2]{\textbf{AH-Pairs}}, we have, for $|(\gamma-\gamma')\log T|\ll M$,
\[ \begin{split} e^{i\alpha(\gamma-\gamma')\log T}K_\lambda((\gamma-\gamma')\log T) &= e^{i\alpha(k\pi + O((|k|+1)R(T))} K_\lambda (k\pi + O((|k|+1)R(T))) \\
&= e^{i\pi k \alpha}\Big(1+ O\bigl(|\alpha|(|k|+1)R(T)\bigr)\Big) \\
&\qquad\times\Big(K_\lambda(k\pi) + O\bigl(\min\{1, \lambda(|k|+1)R(T)\}\bigr)\Big) \\
&= e^{i\pi k \alpha}K_\lambda(k\pi) + O((|\alpha|+\lambda)(|k|+1)R(T)).
\end{split} \]
Thus
\[ G_\lambda(\alpha) = \sum_{\substack{k\in \mathbb{Z} \\ k\ll M}} e^{i\pi k\alpha}\left(\frac{\sin \frac{\lambda\pi k}{2}}{\frac{\lambda\pi k}{2}}\right)^2P_{k/2} + O\left(\frac1{\lambda^2 M}\right)+ O((|\alpha|+ \lambda)M^2R(T)) +O\left(\frac1{\log T}\right). \]
We now extend the summation to all $k\in \mathbb{Z}$ with an error $O(1/\lambda^2M)$, and for simplicity degrade slightly the middle error term to $O((|\alpha|+1)M^2R(T))$. We have defined in \eqref{E_G} these error terms as $E_G(\lambda,\alpha)$ and have obtained \ref{lem4-iii}. 

Clearly \ref{lem4-iv} follows immediately from \ref{lem4-iii}.

\end{proof}

\begin{proof}[Proof of \Cref{Cor5-G(1)ave}] Assuming \hyperref[AH2]{\textbf{AH-Pairs}}, we note for $k\in \mathbb{Z}$ that 
$ \int_0^2 e^{i\pi k \alpha}\, d\alpha = 2 \cdot \1_{k=0} $, and 
therefore from \Cref{lem4-Gfunction} \ref{lem4-iii}
\begin{equation} \label{Cor5-1} \int_0^2 G_\lambda(\alpha)\, d\alpha = 2P_0 + O(E_G(\lambda, 1)).\end{equation}
On the other hand, 
\begin{equation} \begin{split} \label{Cor5-2}
\int_0^2 G_\lambda(\alpha)\, d\alpha &= \left( \int_0^\lambda +\int_\lambda^{1-\lambda}+ \int_{1-\lambda}^{1+\lambda}+\int_{1+\lambda}^{2-\lambda}+\int_{2-\lambda}^2\right)G_\lambda(\alpha)\, d\alpha \\&
= \left( \int_{-\lambda}^\lambda + 2\int_\lambda^{1-\lambda} + \int_{1-\lambda}^{1+\lambda}\right)G_\lambda(\alpha)\, d\alpha + O(E_G(\lambda,1)),
\end{split}\end{equation}
where we have used the evenness and periodicity of $G_\lambda(\alpha)$ from \Cref{lem4-Gfunction} \ref{lem4-iv}.
From \Cref{lem4-Gfunction} \ref{lem4-i} we have assuming RH
\begin{equation} \begin{split} \label{Cor5-3} \int_{-\lambda}^\lambda G_\lambda(\alpha)\, d\alpha &=\int_{-\lambda}^\lambda \frac{\lambda-|\alpha|}{\lambda^2}\, d\alpha + O(\lambda^2)+ O\left(\frac{1}{\sqrt{\log T}}\right)+O\left(\frac{1}{\lambda\log T}\right) \\&
= 1+ O(\lambda^2)+ O\left(\frac{1}{\sqrt{\log T}}\right)+O\left(\frac{1}{\lambda\log T}\right) .
\end{split}\end{equation}
From \Cref{lem4-Gfunction} \ref{lem4-ii} and assuming RH
\begin{equation} \begin{split} \label{Cor5-4} \int_{\lambda}^{1-\lambda} G_\lambda(\alpha)\, d\alpha &= \int_{\lambda}^{1-\lambda}\left(\alpha + O\left(\frac{1}{\sqrt{\log T}}\right) +O\left(\frac{1}{\lambda^2\log T}\right)\right)\, d\alpha \\
&=\frac12 - \lambda + O\left(\frac{1}{\sqrt{\log T}}\right) + O\left(\frac{1}{\lambda^2\log T}\right).
\end{split}\end{equation}
Using \eqref{Cor5-1}--\eqref{Cor5-4} and for simplicity slightly degrading the error terms proves \Cref{Cor5-G(1)ave}.
\end{proof}

\section{Average of $F(\alpha)$ over short intervals}

Notice that $G_\lambda(\alpha)$ in \Cref{lem4-Gfunction} is a smoothed average of $F(\alpha)$ over short intervals. We now prove a lemma that makes the transition to the unweighted average of $F(\alpha)$ over short intervals. Let $[\alpha]$ denote the integer part of $\alpha$, and $\| \alpha \|$ denote the distance from $\alpha$ to the closest integer. Also recall the sawtooth function $s(\alpha)$ from \eqref{s(alpha)}.
 
\begin{lem} \label{lem5-forThm4} Assume the Riemann Hypothesis and \hyperref[AH2]{AH-Pairs}. Then for $L\in \mathbb{Z}$ and $0<\lambda \le 1/4$ we have
 \begin{enumerate}[label={(\roman*)}]
\item \label{lem5-i} $\displaystyle \int_{2L-\lambda}^{2L+\lambda} F(\beta)\, d\beta = 1 + O(\lambda^2)+ O(\lambda E_G(\lambda,L))+ O\left(\frac{1}{\lambda\sqrt{\log T}}\right) $.\\
\end{enumerate}
For any fixed $\alpha \in \mathbb{R}\setminus \mathbb{Z}$, choose $\lambda$ so that $0<\lambda\le \frac12 \| \alpha \|$. Then 
$[\alpha] +2\lambda \le \alpha \le [\alpha]+ 1-2\lambda$, and we have that
\begin{enumerate}[label={(\roman*)}]
\setcounter{enumi}{1}
\item \label{lem5-ii} $\displaystyle \frac{1}{2\lambda}\int_{\alpha-\lambda}^{\alpha +\lambda} F(\beta)\, d\beta = s(\alpha) +O(\lambda) +O\left(\frac{1}{\lambda^5\sqrt{\log T}}\right)+O(\frac{1 }{\lambda}E_G(\lambda, L)) $.\\
\end{enumerate}
Let $K=2L+1$ be an odd integer. Then for $0<\lambda\le 1/4$ we have
\begin{enumerate}[label={(\roman*)}]
\setcounter{enumi}{2}
\item \label{lem5-iii} $\displaystyle \int_{K-\lambda}^{K+\lambda}F(\beta) \,d\beta = 2(P_0-1) + O(\lambda) + O( E_G(\lambda^2,K)) + O\left(\frac{1}{\lambda^4\sqrt{\log T}}\right)$.
\end{enumerate}
If $K=\pm1$, then in addition to \ref{lem5-iii} we have for $0<\lambda\le 1/4$ 
\begin{enumerate}[label={(\roman*)}]
\setcounter{enumi}{3}
\item \label{lem5-iv} $\displaystyle \int_{1-\lambda}^{1+\lambda}F(\beta) \, d\beta = \int_{1}^{1+\lambda}F(\beta) \,d\beta + O(\lambda) ~~\text{and}~~
\int_{-1-\lambda}^{-1+\lambda}F(\beta) \,d\beta = \int_{-1-\lambda}^{-1}F(\beta) \, d\beta + O(\lambda)$. 
\end{enumerate}
For $K$ an odd integer,
\begin{enumerate}[label={(\roman*)}]
\setcounter{enumi}{4}
\item \label{lem5-v} $\displaystyle G_\lambda(K) = \frac{2(P_0-1)}{\lambda} + O(1) + O\left(\frac{E_G(\lambda^2,K)}{\lambda}\right) + O\left(\frac{1}{\lambda^5\sqrt{\log T}}\right)$.
\end{enumerate}
\end{lem} 

With regard to \Cref{lem5-forThm4} \ref{lem5-v}, note by \Cref{lem4-Gfunction} \ref{lem4-i} and \ref{lem4-iv} that for even integers $2K$ we have
\[ G_\lambda(2K) = G_\lambda(0) +O(E_G(\lambda,2K)) = \frac1{\lambda} +O(\lambda) +O(E_G(\lambda,2K))+O\left(\frac{1}{\lambda^2\sqrt{\log T}}\right). \]

For fixed $\alpha$, from \eqref{G_lambda} define
\begin{equation} \label{Hthm4} H_\lambda(\alpha) := \lambda^2 G_\lambda(\alpha) = \int_{-\infty}^\infty k_\lambda(\beta)F(\alpha+\beta)\, d\beta. \end{equation}
For $0<h\le \lambda/2$, we have
\begin{equation}\label{F-Hthm4} \frac{H_\lambda(\alpha) - H_{\lambda -h}(\alpha)}{h} \le 
\int_{-\lambda}^{\lambda}F(\alpha +\beta)\, d\beta \le \frac{H_{\lambda+h}(\alpha)- H_\lambda(\alpha) }{h},\end{equation}
which follows immediately from 
\[ \frac1h\left( k_\lambda(\beta) - k_{\lambda-h}(\beta)\right) \le \1_{[-\lambda,\lambda]}(\beta) \le \frac1h \left( k_{\lambda+h}(\beta) - k_{\lambda}(\beta)\right).\]
\medskip

\begin{proof}[Proof of \Cref{lem5-forThm4} \ref{lem5-i}] From \Cref{lem4-Gfunction} \ref{lem4-i} and \ref{lem4-iv}, we have for fixed $0<\lambda\le 1/4$
\[\begin{split} H_\lambda(2L) &= H_\lambda(0) +O(\lambda^2E_G(\lambda, L))\\& = \lambda^2\left( \frac{1}{\lambda} + O(\lambda)+ O\left(\frac{1}{\lambda\sqrt{\log T}}\right)+O\left(\frac{1}{\lambda^2 \log T}\right) +O(E_G(\lambda,L)) \right) \\&
= \lambda +O(\lambda^3) +O\left(\frac{1}{\sqrt{\log T}}\right) +O(\lambda^2E_G(\lambda,L)),
\end{split} \]
where we have combined and slightly degraded the middle error terms.
Taking $\lambda/4\le h\le \lambda/2$, we have 
\[ \frac{H_{\lambda \pm h}(2L) -H_\lambda(2L)}{\pm h} = 1 + O(\lambda^2)+O(\lambda E_G(\lambda,L))+O\left(\frac{1}{\lambda\sqrt{\log T}}\right). \]
Hence by \eqref{F-Hthm4} we have 
\[ \int_{2L-\lambda}^{2L+\lambda} F(\beta)\, d\beta = 1 + O(\lambda^2)+O(\lambda E_G(\lambda,L))+ O\left(\frac{1}{\lambda\sqrt{\log T}}\right) .\]
 \end{proof}
\begin{proof}[Proof of \Cref{lem5-forThm4} \ref{lem5-ii}]
Take a fixed $\alpha \in \mathbb{R}\setminus \mathbb{Z}$, and choose $\lambda$ so that $0<\lambda\le \frac12 \| \alpha \|$, and thus
\[ [\alpha]+2\lambda \le \alpha \le [\alpha]+ 1-2\lambda. \]
We now define
\begin{equation*} 2L :=
\begin{cases}[\alpha],
 & \text{ if $[\alpha]$ is even,} \\
[\alpha]+1,
 & \text{ if $[\alpha]$ is odd.}
\end{cases}
\end{equation*}
and let $\alpha^*:= \alpha -2L$, so that $-1+2\lambda \le \alpha^*\le - 2\lambda$ if $[\alpha]$ is odd and 
$2\lambda \le \alpha^*\le 1 - 2\lambda$ if $[\alpha]$ is even. 
Then Lemma \ref{lem4-Gfunction} \ref{lem4-iv} and \ref{lem4-ii} implies
\begin{equation}\begin{split} \label{ii)step1} G_\lambda(\alpha) &= G_\lambda(\alpha^*) + O(E_G(\lambda, L)) \\&
= |\alpha^*| +O\left(\frac{1}{\sqrt{\log T}}\right) + O\left(\frac{1}{\lambda^2 \log T}\right) + O(E_G(\lambda, L)),\end{split}\end{equation}
and thus by \eqref{Hthm4} and recalling $s(\alpha)$ from \eqref{s(alpha)},
\[H_\lambda(\alpha) = \lambda^2s(\alpha) +O\left(\frac{\lambda^2}{\sqrt{\log T}}\right)+O\left(\frac{1}{\lambda^2 \log T}\right)+ O(\lambda^2 E_G(\lambda, L)). \]
Hence, taking $0<h\le \lambda/2$, 
\[ \begin{split} \frac{H_{\lambda \pm h}(\alpha) - H_\lambda(\alpha)}{\pm h} &= \frac{((\lambda \pm h)^2-\lambda^2)s(\alpha)}{\pm h} \\
&\qquad+ O\left(\frac{\lambda^2}{h\sqrt{\log T}}\right) + O\left(\frac{1}{h\lambda^2 \log T}\right) + O\left(\frac{\lambda^2}{h}E_G(\lambda,L)\right) \\&
= 2\lambda s(\alpha) + O(h) +O\left(\frac{\lambda^2}{h\sqrt{\log T}}\right) + O\left(\frac{1}{h\lambda^2 \log T}\right) + O\left(\frac{\lambda^2 }{h}E_G(\lambda,L)\right).
\end{split}\]
We conclude by \eqref{F-Hthm4} that
\[\frac{1}{2\lambda}\int_{\alpha-\lambda}^{\alpha +\lambda} F(\beta)\, d\beta = s(\alpha) + O(\frac{h}{\lambda}) + O\left(\frac{\lambda}{h\sqrt{\log T}}\right) + O\left(\frac{1}{h\lambda^3 \log T}\right) + O\left(\frac{\lambda }{h}E_G(\lambda,L)\right). \]
We complete the proof by choosing $h=\lambda^2$ and degrade slightly the two middle error terms to obtain
\[\frac{1}{2\lambda}\int_{\alpha-\lambda}^{\alpha +\lambda} F(\beta)\, d\beta = s(\alpha) + O(\lambda) + O\left(\frac{1}{\lambda^5\sqrt{\log T}}\right) + O\left(\frac{1}{\lambda}E_G(\lambda, L)\right). \]
\end{proof}

\begin{proof}[Proof of \Cref{lem5-forThm4} \ref{lem5-iii}, \ref{lem5-iv}, and \ref{lem5-v}] The steps in this proof are to first prove \ref{lem5-iii} when $K=1$, which then almost immediately gives \ref{lem5-iv}. Using this result we prove \ref{lem5-v} when $K=1$. Next, by \Cref{lem4-Gfunction} \ref{lem4-iv} we obtain \ref{lem5-v} for all odd $K$, and then we obtain from this \ref{lem5-iii} for all odd $K$. The idea here is the average of $G_\lambda(\alpha)$ in the interval $(1-\lambda, 1+\lambda)$ is just the unweighted average of $F(\alpha)$ in $(1-2\lambda,1+2\lambda)$ which we then obtain from \Cref{Cor5-G(1)ave}, which then allows us to obtain $G_\lambda(1)$. Next, we obtain $G_\lambda(K)$ immediately from \Cref{lem4-Gfunction} \ref{lem4-iv}, and then a differencing argument recovers the unweighted average of $F(\alpha)$ in an interval around $K$.

We first recall from \Cref{Cor5-G(1)ave} that
\[\int_{1-\lambda}^{1+\lambda}G_{\lambda}(\alpha) \, d\alpha = 2(P_0-1) +O(\lambda)+ O( E_G(\lambda,1)) +O\left(\frac{1}{\lambda^2\sqrt{\log T}}\right).\]
By \eqref{G_lambda}
\[ \begin{split} 
\int_{1-\lambda}^{1+\lambda} G_\lambda(\alpha)\, d\alpha &= \int_{1-\lambda}^{1+\lambda}\frac{1}{\lambda^2}\int_{-\lambda}^\lambda (\lambda -|\beta|)F(\alpha +\beta)\, d\beta \,d\alpha \\ &
= \frac{1}{\lambda^2}\int_{1-\lambda}^{1+\lambda}\int_{\alpha-\lambda}^{\alpha+\lambda}(\lambda -|w-\alpha|)F(w)\, dw \, d\alpha \\&
= \frac{1}{\lambda^2}\int_{1-2\lambda}^{1+2\lambda}F(w)\left(\int_{w-\lambda}^{w+\lambda}(\lambda -|w-\alpha|) \, d\alpha\right) \, dw \\&
=\int_{1-2\lambda}^{1+2\lambda}F(w)\, dw,
\end{split} \]
and therefore
\[ \int_{1-2\lambda}^{1+2\lambda}F(w)\, dw = 2(P_0-1) + O(\lambda) + O(E_G(\lambda,1)) +O\left(\frac{1}{\lambda^2\sqrt{\log T}}\right).\]
On replacing $2 \lambda$ by $\lambda$ we obtain a slightly stronger version of \ref{lem5-iii} when $K=1$ since $E_G(\lambda,\alpha)$ is a decreasing function of $\lambda$. By \hyperref[MonThm]{\textbf{MT}} we have $F(w)\ll 1$ for $1-\lambda \le w\le 1$ and therefore
\[ \int_{1-\lambda}^{1+\lambda}F(w)\, dw = \int_{1}^{1+\lambda}F(w)\, dw + O(\lambda),\]
which proves \ref{lem5-iv} when $K=1$. We thus conclude that 
\begin{equation} \label{K=1iii} \int_{1}^{1+\lambda}F(w)\, dw = 2(P_0-1) + O(\lambda) + O(E_G(\lambda,1)) +O\left(\frac{1}{\lambda^2\sqrt{\log T}}\right). \end{equation}
Similarly for $K=-1$ by \hyperref[MonThm]{\textbf{MT}} the corresponding result in \Cref{lem5-forThm4} \ref{lem5-iv} also follows.

We now prove \ref{lem5-v} for $K=1$. Taking $0<\eta \le \lambda/2$, we have as above by \eqref{G_lambda}, \hyperref[MonThm]{\textbf{MT}}, and $F\ge 0$ that 
\[\begin{split} G_\lambda(1) &= 
\frac{1}{\lambda^2}\int_{1-\lambda}^{1+\lambda}(\lambda -|w-1|)F(w)\, dw \\&
= \frac{1}{\lambda^2}\int_1^{1+\lambda}(\lambda -|w-1|)F(w)\, dw +O(1) \\&
= \frac{\lambda+O(\eta)}{\lambda^2}\int_{1}^{1+\eta}F(w)\, dw + O\left(\frac{1}{\lambda }\int_{1+\eta}^{1+\lambda}F(w)\, dw 
 \right)+O(1)\\&
= \frac{\lambda+O(\eta)}{\lambda^2}\int_{1}^{1+\eta}F(w)\, dw + O\left(\frac{1}{\lambda }\left(\int_1^{1+\lambda}F(w)\, dw -\int_1^{1+\eta}F(w)\, dw \right)\right)+O(1).
\end{split}\]
Using \eqref{K=1iii} and that $E(\lambda,\alpha)$ is a decreasing function of $\lambda$, we have
\[\begin{split} G_\lambda(1) &= \frac{\lambda+O(\eta)}{\lambda^2}\left(2(P_0-1) + O(\eta)+ O( E_G(\eta,1)) +O\left(\frac{1}{\eta^2 \sqrt{\log T}}\right)\right) \\& \qquad + O\left(\frac{1}{\lambda }\left( \lambda+ E_G(\eta,1) +\frac{1}{\eta^2\sqrt{\log T}}\right)\right)+O(1) \\&
= \frac{2(P_0-1)}{\lambda} + O\left(\frac{\eta}{\lambda^2}\right) + O(1) + O\left(\frac{E_G(\eta,1)}{\lambda} \right) +O\left(\frac{1}{\eta^2\lambda\sqrt{\log T}}\right).\end{split} \]
With the choice $\eta = \lambda^2$, we have
\[ G_\lambda(1)= \frac{2(P_0-1)}{\lambda} + O(1) + O\left(\frac{E_G(\lambda^2,1)}{\lambda} \right) +O\left(\frac{1}{\lambda^5\sqrt{\log T}}\right),\]
which proves \ref{lem5-v} when $K=1$. By \Cref{lem4-Gfunction} \ref{lem4-iv} 
\[ \begin{split} G_\lambda(K)&= G_\lambda(2L+1) = G_{\lambda}(1) +O(E_G(\lambda,K))\\&
= \frac{2(P_0-1)}{\lambda} + O(1) + O\left(\frac{E_G(\lambda^2,K)}{\lambda} \right) + O\left(\frac{1}{\lambda^5\sqrt{\log T}}\right),
\end{split}\]
which completes the proof of \ref{lem5-v}. 

It remains to prove \ref{lem5-iii} when $K\neq 1$. From the last equation and \eqref{Hthm4} we have 
\[H_\lambda(K) = \lambda^2 G_\lambda(K) = 2(P_0-1)\lambda +O(\lambda^2) + O(\lambda E_G(\lambda^2,K)) +O\left(\frac{1}{\lambda^3\sqrt{\log T}}\right), \]
and therefore, for $0< h \le \lambda/2$,
\[ \frac{H_{\lambda \pm h}(K) -H_\lambda(K)}{\pm h} = 2(P_0-1) +O\left(\frac{\lambda^2}{h}\right) + O\left(\frac{\lambda}{h} E_G(\lambda^2,K)\right) + O\left(\frac{1}{h\lambda^3\sqrt{\log T}}\right).\]
Hence, choosing $h=\lambda/2$, we conclude by \eqref{F-Hthm4} that
\[ \int_{-\lambda}^{\lambda} F(K+\beta)\, d\beta = 2(P_0-1) +O(\lambda) + O( E_G(\lambda^2,K)) + O\left(\frac{1}{\lambda^4\sqrt{\log T}}\right), \]
which completes the proof.
\end{proof}

\section{Proof of \Cref{thm4-r(k)p_k} and \Cref{cor4}}

\begin{proof}[Proof of \Cref{thm4-r(k)p_k}] In this proof we are taking $T \to \infty$ and then making $M$ large and $\lambda$ appropriately small in \Cref{lem5-forThm4}.
First, using \Cref{lem5-forThm4} \ref{lem5-i} and \ref{lem5-iii}, we have for $L\in \mathbb{Z}$ and the definition of $\mathcal{F}$ in \eqref{calF},
\[ \int_{2L-\lambda}^{2L+\lambda}F(\alpha)\, d\alpha = 1 +o(1) \sim \int_{2L-\lambda}^{2L+\lambda} \delta_{2L}(\alpha)\, d\alpha \sim \int_{2L-\lambda}^{2L+\lambda}\mathcal{F}(\alpha)\, d\alpha ,\]
and similarly with $K=2L+1$
\[ \int_{K-\lambda}^{K+\lambda}F(\alpha)\, d\alpha = 2(P_0-1) +o(1) \sim 2(P_0-1)\int_{K-\lambda}^{K+\lambda} \delta_{K}(\alpha)\, d\alpha \sim \int_{K-\lambda}^{K+\lambda}\mathcal{F}(\alpha)\, d\alpha .\]
We now consider the results with the one-sided delta functions contained in $\mathcal{F}$. For $L=0$, we have, since $F$ is even, 
\[\int_{-\lambda}^\lambda F(\beta)\, d\beta = 
2 \int_{-\lambda}^0 F(\beta)\, d\beta =2\int_{0}^\lambda F(\beta)\, d\beta ,\]
and assuming only RH we apply \hyperref[MonThm]{\textbf{MT}} and have
\[\int_{0}^\lambda F(\beta)\, d\beta = \frac12 +O(\lambda) +o(1) = \int_0^\lambda \delta_0^{=}(\beta)\, d\beta +O(\lambda) +o(1),\] and similarly for the left-side integral. 
Similarly, for $K=\pm 1$, the results with the one-sided delta functions contained in $\mathcal{F}$
follow immediately from \Cref{lem5-forThm4} \ref{lem5-iii} and \ref{lem5-iv}. 
 
Next consider a function $g(\alpha)$ which is Lipschitz continuous at $\alpha =N\in \mathbb{Z}$. Then 
\[ \int_{N-\lambda}^{N+\lambda}F(\alpha)g(\alpha)\, d\alpha \sim g(N) \int_{N-\lambda}^{N+\lambda} \mathcal{F}(\alpha)\, d\alpha , \]
and the corresponding one-sided results hold when $N=0,\pm 1$.

Next, by \Cref{lem5-forThm4} \ref{lem5-ii}
for an interval $[a-\lambda,a+\lambda]\subset (N+\lambda,N+1-\lambda) $ for $N\in \mathbb{Z}$, we have 
\[ \int_{a-\lambda}^{a+\lambda}F(\alpha)\, d\alpha \sim \int_{a-\lambda}^{a+\lambda} s(\alpha) \, d\alpha \sim \int_{a-\lambda}^{a+\lambda} \mathcal{F}(\alpha)\, d\alpha . \]
Now suppose $[a,b]\subset (N+\lambda,N+1-\lambda) $ for $N\in \mathbb{Z}$. For a step function $r(\alpha)$ in this interval, we pick $\lambda$ smaller than the mesh of this step function and can write $r$ as a step function over intervals of the form $(c_j-\lambda, c_j+\lambda)$, and conclude
\[ \int_a^b F(\alpha) r(\alpha) \, d\alpha \sim \int_a^b s(\alpha) r(\alpha) \, d\alpha \sim \int_a^b \mathcal{F}(\alpha) r(\alpha) \, d\alpha .\]
Thus by approximation this also holds for any Riemann integrable function $g(\alpha)$.

Now for a general interval $[a,b]$, $a,b\not \in \mathbb{Z}$, we have writing $[a] = N-1$, $[b] =M$, and taking $\lambda < \min(\|a\|,\|b\|)$,
\begin{align*}
\int_a^bF(\alpha)g(\alpha)\, d\alpha &= \int_a^{N-\lambda}Fg + \sum_{N\le j \le M} \int_{j-\lambda}^{j+\lambda}Fg + \sum_{N\le j \le M-1} \int_{j+\lambda}^{j+1-\lambda}Fg + \int_{M+\lambda}^{b}Fg \\
&\sim \int_a^b\mathcal{F}(\alpha)g(\alpha)\, d\alpha,
\end{align*}
since each integral in the middle step has been evaluated above to give the right-hand side. If $a$ or $b$ are equal to $0$ or $\pm 1$ we use the one-sided delta functions to obtain the same result. 
\end{proof}
\begin{proof}[Proof of Corollary \ref{cor4}]
Assuming Montgomery's conjecture, we have for any large $U>0$,
\begin{align*}
\int_1^\infty \frac{F(\alpha)}{\alpha^2}\, d\alpha
= \int_1^U \frac{1+o(1)}{\alpha^2}\, d\alpha + \int_U^\infty \frac{F(\alpha)}{\alpha^2}\, d\alpha.
\end{align*}
By \cite[Lemma A]{Goldston87}, we have uniformly for all $a=a(T)$,
\[\int_a^{a+1} F(\alpha)\, d\alpha \ll 1.\] 
Hence we have
$$ 0\le \int_U^\infty \frac{F(\alpha)}{\alpha^2}\, d\alpha
\le \sum_{n=[U]}^\infty \frac1{n^2} \int_n^{n+1} F(\alpha)\, d\alpha
\ll \sum_{n=[U]}^\infty \frac1{n^2} \ll \frac1U.
$$
and therefore
$$ \int_1^\infty \frac{F(\alpha)}{\alpha^2}\, d\alpha
= 1 + O(U^{-1}) + o(1), $$
which upon letting $U=U(T)\to\infty$ as $T\to \infty$ gives
$$ \int_1^\infty \frac{F(\alpha)}{\alpha^2}\, d\alpha
= 1 + o(1). $$
Thus 
$$ \mathcal{C} = \lim_{T\to\infty} \int_1^\infty \frac{F(\alpha)}{\alpha^2}\, d\alpha = 1. $$

Now assume RH and \hyperref[AH2]{\textbf{AH-Pairs}}. We obtain as above
\begin{equation} \label{int-F-asymp}
\int_1^\infty \frac{F(\alpha)}{\alpha^2}\, d\alpha = \int_1^U \frac{F(\alpha)}{\alpha^2}\, d\alpha + O(U^{-1})
\sim \int_1^U\frac{\mathcal{F}(\alpha)}{\alpha^2}\, d\alpha + O(U^{-1}).
\end{equation}
Applying \Cref{thm4-r(k)p_k} to the interval $[1,U]$, with $U\not \in \mathbb{Z}$, we have
\begin{equation} \label{int-mathcal-F}
\begin{aligned}
\int_1^U \frac{\mathcal{F}(\alpha)}{\alpha^2}\, d\alpha &= \sum_{\substack{n\\ 2n+1<U}}
\left( \int_{2n-1}^{2n} \frac{2n-\alpha}{\alpha^2}\, d\alpha
+ \int_{2n}^{2n+1} \frac{\alpha-2n}{\alpha^2}\, d\alpha
+ \frac{2(P_0-1)}{(2n-1)^2} + \frac1{(2n)^2}\right) \\
&\qquad+ O(U^{-2}) \\
&= \sum_{n=1}^\infty
\left( \int_{2n-1}^{2n} \frac{2n}{\alpha^2}\, d\alpha
- \int_{2n}^{2n+1} \frac{2n}{\alpha^2}\, d\alpha
-\log\frac{2n}{2n-1} + \log\frac{2n+1}{2n} \right) \\
&\qquad+ 2(P_0-1)\sum_{n=1}^\infty\frac1{(2n-1)^2} + \frac14\zeta(2) + O(U^{-1}) \\
&= \sum_{n=1}^\infty
\left( \frac1{2n-1} - \frac1{2n+1}\right)+\sum_{n=1}^\infty \log\left(1-\frac1{4n^2}\right)
+ \left(\frac32(P_0-1)+\frac14\right)\zeta(2) \\
&\qquad+ O(U^{-1}) \\
&= 1 + \left(\frac32(P_0-1)+\frac14\right)\frac{\pi^2}6 + \sum_{n=1}^\infty \log\left(1-\frac1{4n^2}\right) + O(U^{-1}).
\end{aligned}
\end{equation}
Recalling
$$ \frac{\sin{z}}z = \prod_{n=1}^\infty \left(1-\frac{z^2}{(n\pi)^2}\right), $$
we have
\begin{align*}
\sum_{n=1}^\infty \log\left(1-\frac1{4n^2}\right)
= \log \prod_{n=1}^\infty \left(1-\frac1{4n^2}\right)
= \log \prod_{n=1}^\infty \left(1-\frac{(\pi/2)^2}{n^2\pi^2}\right)
= \log\frac2\pi.
\end{align*}
Hence substituting this into \eqref{int-F-asymp} and \eqref{int-mathcal-F}, we obtain
$$ \int_1^\infty \frac{F(\alpha)}{\alpha^2}\, d\alpha
= 1 + \left(\frac32(P_0-1)+\frac14\right)\frac{\pi^2}6 + \log\frac2\pi +o(1)$$
on letting $U=U(T)\to \infty$ as $T\to\infty$.
If the limiting density
$$ p_0 = \lim_{T\to\infty} P_0(T) $$
exists, then the above asymptotic implies that
$$ \mathcal{C} = \lim_{T\to\infty} \int_1^\infty \frac{F(\alpha)}{\alpha^2}\, d\alpha
= 1 + \left(\frac32(p_0-1)+\frac14\right)\frac{\pi^2}6 + \log\frac2\pi. $$
\end{proof}

\section{Proof of \Cref{thm2-generalizedMTwith_p_k}}

\begin{proof}[Proof of \Cref{thm2-generalizedMTwith_p_k}] Let $\widetilde{p}_0 = p_0-1$. Then we see by \hyperref[AH-density]{\textbf{AH-Density}} that
\[ \begin{split} \mathcal{S}(r):=\sum_{k\in \mathbb{Z}} \widehat{r}(k/2)p_{k/2} & = \sum_{\substack{k\in \mathbb{Z}\\ k \text{\ even}}}\left(p_0-\frac12\right) \widehat{r}(k/2)+ \sum_{\substack{k\in \mathbb{Z}\\ k \text{\ odd}}}\left(\frac32 - \frac{2}{\pi^2k^2}-p_0\right) r(k/2)\\ &
= \widetilde{p}_0\left( \sum_{\substack{k\in \mathbb{Z}\\ k \text{\ even}}} \widehat{r}(k/2) - \sum_{\substack{k\in \mathbb{Z}\\ k \text{\ odd}}} \widehat{r}(k/2)\right) \\
&\qquad+ \left( \frac12 \sum_{k\in \mathbb{Z}} \widehat{r}(k/2) - \frac{2}{\pi^2}\sum_{\substack{k\in \mathbb{Z}\\ k \text{\ odd}}}\frac{1}{k^2} \widehat{r}(k/2)\right).
\end{split} \]
We first use the Poisson Summation Formula (PSF) to show that
\begin{equation} \label{even=odd} \sum_{\substack{k\in \mathbb{Z}\\ k \text{\ even}}} \widehat{r}(k/2) = \sum_{\substack{k\in \mathbb{Z}\\ k \text{\ odd}}} \widehat{r}(k/2), \end{equation} 
and therefore 
\begin{equation} \label{S(r)1} \mathcal{S}(r) = \frac12 \sum_{k\in \mathbb{Z}} \widehat{r}(k/2) - \frac{2}{\pi^2}\sum_{\substack{k\in \mathbb{Z}\\ k \text{\ odd}}}\frac{1}{k^2} \widehat{r}(k/2).\end{equation}
The PSF states that with certain conditions on $f \in L^1(\mathbb{R})$ we have, for any real number $\lambda >0$, that
\begin{equation} \label{PSF}
\sum_{k\in \mathbb{Z}} \lambda f(\lambda k) = \sum_{k\in \mathbb{Z}} \widehat{f}(k/\lambda). \end{equation}
We use \cite[Appendix D.2]{MontgomeryVaughan2007} from which it follows that PSF holds as stated above if $f \in L^1(\mathbb{R})$ is even and of bounded variation, and $\widehat{f} \in L^1(\mathbb{R})$.\footnote[2]{It is shown in \cite[Ch 6, Section 1, Pr. 15]{Katznelson76} that PSF may not hold even when both $f, \widehat{f} \in L^1(\mathbb{R})$ and thus $f,\widehat{f}$ are both continuous and both series in PSF are absolutely convergent.} Thus, taking $\lambda = 1$ in \eqref{PSF}, we have
\[ \sum_{k\in \mathbb{Z}} \widehat{r}( k) = \sum_{k\in \mathbb{Z}} r(k). \]
Since $r(\alpha)$ is continuous and supported in $|\alpha|\le 1$, we have $r(k)=0$ for $k\neq 0$, and therefore
\begin{equation*} \sum_{\substack{k\in \mathbb{Z}\\ k \text{\ even}}} \widehat{r}(k/2) = \sum_{k\in \mathbb{Z}} \widehat{r}( k) = r(0).\end{equation*}
Next, taking $\lambda=2$ in \eqref{PSF}, we have
\begin{equation*} \sum_{\substack{k\in \mathbb{Z}\\ k \text{\ even}}} \widehat{r}(k/2)+\sum_{\substack{k\in \mathbb{Z}\\ k \text{\ odd}}} \widehat{r}(k/2) = 
\sum_{k\in \mathbb{Z}} \widehat{r}(k/2) = 2 \sum_{k\in \mathbb{Z}} r(2k) = 2 r(0),
\end{equation*}
and \eqref{even=odd} follows from these last two equations.

We now expand $r(\alpha)$ formally into a Fourier series in $|\alpha|\le 1$. Thus 
\[ r(\alpha) = \sum_{k\in \mathbb{Z}} a_k e(k\alpha/2) , \qquad \text{for $|\alpha |\le 1$ } ,
\]
where 
\[ a_k = \frac12\int_{-1}^1r(\alpha) e(-k\alpha/2)\, d\alpha = \frac12\int_{-\infty}^\infty r(\alpha) e(-k\alpha/2)\, d\alpha = \frac12 \widehat{r}(k/2).\]
Thus we have that 
\begin{equation*} r(\alpha) =\frac12 \sum_{k\in \mathbb{Z}} \widehat{r}(k/2) e(k\alpha/2) , \qquad \text{for $|\alpha |\le 1$ }, \end{equation*}
where the series converges to $r(\alpha)$ by bounded variation. Thus
\[ \begin{split} r(0) +2 \int_0^1\alpha r(\alpha)\, d\alpha &= \frac12 \sum_{k\in \mathbb{Z}} \widehat{r}(k/2) + \sum_{k\in \mathbb{Z}} \widehat{r}(k/2)\operatorname{Re}\int_0^1 \alpha e(k\alpha/2)\, d\alpha \\&
= \frac12 \sum_{k\in \mathbb{Z}} \widehat{r}(k/2) - \frac{2}{\pi^2}\sum_{\substack{k\in \mathbb{Z}\\ k\ \text{odd}}}\frac{1}{k^2} \widehat{r}(k/2),
\end{split} \]
which by \eqref{S(r)1} completes the proof.
\end{proof}

\section*{Acknowledgment and Funding}

The authors thank the American Institute of Mathematics for their hospitality and for providing a pleasant environment where work on this research began.
The third author was supported by JSPS KAKENHI Grant Numbers 18K13400 and 22K13895, and also by MEXT Initiative for Realizing Diversity in the Research Environment during the beginning of this work, and later was also supported by the Inamori Research Grant 2024 when we completed this work.
The fourth author was partially supported by NSF DMS-1902193 and NSF DMS-1854398 FRG during the beginning of this work and is currently partially supported by NSF CAREER DMS-2239681.

\bibliographystyle{alpha}
\bibliography{AHReferences}

\end{document}